\documentclass[reqno,10pt,centertags]{amsart} 
\usepackage{amsmath,amsthm,amscd,amssymb,latexsym,esint,upref,stmaryrd,
enumerate,verbatim,yfonts,bm,mathtools}
\usepackage{color}
\usepackage{hyperref} 

\usepackage{amsfonts}
\usepackage{amsmath}
\usepackage{amssymb}
\usepackage{graphicx}
\usepackage{verbatim}
\usepackage[export]{adjustbox}
\usepackage{subcaption}
\usepackage{booktabs} 
\usepackage{enumerate}
\usepackage[export]{adjustbox}
\usepackage{pgfplots}
\pgfplotsset{compat=1.12}
\usepackage{tikz}
\usetikzlibrary{spy}
\usepackage{comment}
\usepackage{animate}
\usepackage[absolute, overlay]{textpos}
\usepackage{pdfpages}
\usepackage{float}

\newcommand{\C}{{\mathbb C}}

\newcommand{\bbC}{{\mathbb{C}}}

\newcommand{\bbN}{{\mathbb{N}}}

\newcommand{\bbR}{{\mathbb{R}}}

\newcommand{\cB}{{\mathcal B}}

\newcommand{\cH}{{\mathcal H}}

\newcommand{\cM}{{\mathcal M}}

\newcommand{\beq}{\begin{equation}}
\newcommand{\enq}{\end{equation}}

\DeclareMathOperator{\ran}{ran}
\DeclareMathOperator{\dom}{dom}

\DeclareMathOperator*{\slim}{s-lim}

\renewcommand{\Re}{\text{\rm Re}}

\renewcommand{\ln}{\text{\rm ln}}

\newcommand{\no}{\notag}
\newcommand{\lb}{\label}
\newcommand{\f}{\frac}

\newcommand{\ol}{\overline}

\newcommand{\Oh}{O}

\newcommand{\hatt}{\widehat} 
\newcommand{\bi}{\bibitem}

\renewcommand{\ge}{\geqslant}

\let\geq\geqslant
\let\leq\leqslant

\makeatletter
\def\theequation{\@arabic\c@equation}

\allowdisplaybreaks 
\numberwithin{equation}{section}

\newtheorem{theorem}{Theorem}[section]
\newtheorem{proposition}[theorem]{Proposition}
\newtheorem{lemma}[theorem]{Lemma}
\newtheorem{corollary}[theorem]{Corollary}

\theoremstyle{remark}
\newtheorem{remark}[theorem]{Remark}

\begin{document}

\title[Hardy--Rellich-Type Inequalities]{On Birman's Sequence of Hardy--Rellich-Type Inequalities}
  
\author[F.\ Gesztesy]{Fritz Gesztesy}
\address{Department of Mathematics, 
Baylor University, One Bear Place \#97328,
Waco, TX 76798-7328, USA}
\email{Fritz$\_$Gesztesy@baylor.edu}
\urladdr{http://www.baylor.edu/math/index.php?id=935340}

\author[L.\ L.\ Littlejohn]{Lance L. Littlejohn}
\address{Department of Mathematics, 
Baylor University, One Bear Place \#97328,
Waco, TX 76798-7328, USA}
\email{Lance$\_$Littlejohn@baylor.edu}
\urladdr{http://www.baylor.edu/math/index.php?id=53980}

\author[I.\ Michael]{Isaac Michael}
\address{Department of Mathematics, 
Baylor University, One Bear Place \#97328,
Waco, TX 76798-7328, USA}
\email{Isaac$\_$Michael@baylor.edu}
\urladdr{http://blogs.baylor.edu/isaac\_michael/}

\author[R.\ Wellman]{Richard Wellman}
\address{Department of Mathematics, 
Department of Mathematics, Westminster College, Foster Hall, Salt Lake City,
UT 84105-3617, USA}
\email{rwellman@westminstercollege.edu}
\urladdr{http://www.baylor.edu/math/index.php?id=935340}

\dedicatory{Dedicated with great pleasure to Eduard Tsekanovskii on the occasion of his 
80th birthday.}

\date{\today}
\thanks{Appeared in {\it J. Diff. Eq.} {\bf 264}, 2761--2801 (2018).}

\begin{abstract} 
In 1961, Birman proved a sequence of inequalities $\{I_{n}\}$, for
$n\in\bbN$, valid for functions in $C_0^{n}((0,\infty))\subset
L^{2}((0,\infty)).$ In particular, $I_{1}$ is the classical (integral) Hardy inequality and
$I_{2}$ is the well-known Rellich inequality. In this paper, we give a
proof of this sequence of inequalities valid on a certain Hilbert space
$H_{n}([0,\infty))$ of functions defined on $[0,\infty).$ Moreover, $f\in H_{n}([0,\infty))$
implies $f^{\prime}\in H_{n-1}([0,\infty))$; as a consequence of this inclusion, we
see that the classical Hardy inequality implies \emph{each }of the inequalities
in Birman's sequence. We also show that for any finite $b>0,$ these inequalities hold 
on the standard Sobolev space $H_0^{n}((0,b))$. Furthermore, in all cases, the Birman constants 
$[(2n-1)!!]^{2}/2^{2n}$ in these
inequalities are sharp and the only function that gives equality in any of
these inequalities is the trivial function in $L^{2}((0,\infty))$ (resp., $L^2((0,b))$). We also show
that these Birman constants are related to the norm of a generalized continuous 
Ces\`aro averaging operator whose spectral properties we determine in detail.
\end{abstract}

\maketitle


{\scriptsize{\tableofcontents}}

\section{Introduction} \lb{s1}

In 1961, M. \v{S}. Birman \cite[p.~48]{Bi66}, sketched a proof to establish the
following sequence of inequalities 
\begin{equation}
\int_{0}^{\infty}\big| f^{(n)}(x)\big|^{2}dx\geq\frac{[(2n-1)!!]^{2}}{2^{2n}}
\int_{0}^{\infty} \frac{|f(x)|^2}{x^{2n}} \, dx,  \quad n\in\bbN  , \lb{1.1} 
\end{equation}
valid for $f\in C_0^{n}((0,\infty)),$ the space of $n$-times continuously
differentiable complex-valued functions having compact support on
$(0,\infty).$ Here we employed the well-known symbol, $(2n-1)!!:=(2n-1)\cdot
(2n-3)\cdots3\cdot1.$ We denote the inequality in \eqref{1.1}
by $I_{n}.$ In particular, $I_{1}$ is the classical (integral) Hardy inequality (see
\cite[Sect.~7.3]{HLP88}) 
\begin{equation}
\int_{0}^{\infty}\left\vert f^{\prime}(x)\right\vert ^{2}dx\geq\frac{1}{4} 
\int_{0}^{\infty} \frac{|f(x)|^2}{x^2} \, dx, 
\lb{1.2} 
\end{equation}
and $I_{2}$ is the Rellich inequality 
\begin{equation}
\int_{0}^{\infty}\left\vert f^{\prime\prime}(x)\right\vert ^{2}dx\geq\frac
{9}{16}\int_{0}^{\infty} \frac{|f(x)|^2}{x^{4}} \, dx.
\lb{1.3} 
\end{equation}
We can find no reference in the literature to the general inequality
\eqref{1.1} prior to the 1966 work of Birman cited above. In
\cite[pp.~83--84]{Gl66}, Glazman gives a detailed proof of
\eqref{1.1} using the ideas outlined in \cite{Bi66}. In
\cite[Lemma~2.1]{Ow99}, Owen also establishes these inequalities. Each of these
authors prove \eqref{1.1} for functions on $C_0^{n} (0,\infty).$ We note in passing that 
unless $f \equiv 0$, all inequalities \eqref{1.1}--\eqref{1.3} are strict. 

In this paper we offer a new proof of \eqref{1.1} and
confirm that the constant $[(2n-1)!!]^{2}/2^{2n}$ is best
possible. We establish these inequalities for a general class of functions
defined on $[0,\infty);$ the significance of this class is that we address the
singularity at $x=0$, which is apparent on the right-hand side of
\eqref{1.1}, rather than deal with functions from $C_0 ^{n}((0,\infty)).$ 
More specifically, we prove the inequalities in
\eqref{1.1} are valid for all functions $f\in H_{n}([0,\infty)),$ for
$n\in\bbN  $, where 
\begin{align}
\begin{split} 
 H_{n}([0,\infty)):= \big\{f:[0,\infty)\rightarrow\bbC  \, \big| \,  f^{(j)}\in 
AC_{loc}([0,\infty)); \, f^{(n)}\in L^{2}((0,\infty)); \lb{1.4}&  \\
 f^{(j)}(0)=0, \, j=0,1,\dots,n-1&\big\}. 
\end{split} 
\end{align}
In \cite[Sect.~7.3]{HLP88}, Hardy, Littlewood, and P{\' o}lya established the classical Hardy
inequality \eqref{1.2} on $H_{1}.$ As we will see, the space
$H_{n}([0,\infty))$ is a Hilbert space when endowed with the inner product 
\begin{equation}
(f,g)_{H_{n}([0,\infty))}:=\int_{0}^{\infty} \ol{f^{(n)}(x)} \, g^{(n)}(x) \, dx, \quad f,g\in
H_{n}([0,\infty)). \lb{1.5} 
\end{equation}
We also show that 
\begin{equation}
H_{n}([0,\infty))= \hatt H_{n}((0,\infty)), \lb{1.6} 
\end{equation}
where\footnote{We emphasize from the outset that despite the similarity of notation with Sobolev spaces, neither of the spaces $H_{n}([0,\infty))$ nor $\hatt H_{n}((0,\infty))$ coincides with the standard Sobolev space $H^{n}_0((0,\infty))$. (See, however, Theorem \ref{t8.2} in the finite interval context.)} 
\begin{align}
& \hatt H_{n}((0,\infty)):= \big\{f:(0,\infty)\rightarrow\bbC \, \big| \, f^{(j)}\in
AC_{loc}((0,\infty)), \, j=0,1,\dots,n-1;  \no \\
& \hspace*{7cm} f^{(n)},f/x^{n}\in L^{2}((0,\infty))\big\}. \lb{1.7} 
\end{align}
Upon first glance, it may seem unlikely that these spaces can be equal
since one set deals with functions defined on $[0,\infty)$ while the other set
has its functions defined on $(0,\infty).$ However, we will show that
functions $f\in \hatt H_{n}((0,\infty))$, and their derivatives $f^{(j)}$ when
$j=0,1,\dots,n-1,$ will have finite limits $f^{(j)}(0_{+})$ at $x=0$.

There is an interesting connection with the spaces $\{H_{n}([0,\infty))\}_{n \in \bbN}$,
namely that
\begin{equation}
f\in H_{n}([0,\infty))\, \text{ implies } \,  f^{\prime}\in H_{n-1}([0,\infty)); \lb{1.8} 
\end{equation}
this inclusion is important in establishing a new proof of
\eqref{1.1} and in proving when equality in
\eqref{1.1} occurs. Moreover, we will show that, in a sense, \textit{each} of the 
inequalities $I_{n},$ $n\in\bbN,$ follows from the classical Hardy inequality $I_{1}.$

In Section \ref{s2}, we discuss a theorem attributed to a
number of mathematicians, including Talenti, Tomaselli, Chisholm and Everitt,
and Muckenhoupt; this result is useful in establishing various properties of
functions in the spaces $H_{n}([0,\infty))$. These properties are dealt with in Section
\ref{s3} where we establish the identity in \eqref{1.6}. In Section \ref{s4}, besides
giving a slight extension of Glazman's proof of \eqref{1.1} including a power weight, 
we offer a new proof (Theorem \ref{t4.4}) of \eqref{1.1} on the set
$H_{n}([0,\infty))= \hatt H_{n}((0,\infty)).$ We will see, from \eqref{1.8}, that each of
these inequalities, when considered on $H_{n}([0,\infty))$, follows in a sense from the
classical Hardy inequality \eqref{1.2}. While the inequality in \eqref{1.1} certainly does not 
imply that the Birman constant $[(2n-1)!!]^2/2^{2n}$ is sharp, the latter fact is well-known 
and in Section \ref{s6}, we confirm that this constant is best possible in
$H_{n}([0,\infty))$. In Section \ref{s7}, we connect the Birman
constants to the norms of generalized continuous Ces\`aro averaging operators $T_n$, 
$n \in \bbN$, and 
determine their spectra; in Section \ref{s8}, we discuss the Birman inequalities on the
finite interval $[0,b]$, $b \in (0, \infty).$ Finally, in Section \ref{s9} we derive Birman's 
sequence of Hardy--Rellich-type inequalities in the vector-valued case replacing 
complex-valued $f(x)$ by $f(x) \in \cH$, with $\cH$ a complex, separable Hilbert space. 

Finally, a few comments on the notation used in this paper: $AC_{loc}((a,b))$ denotes 
the functions locally absolutely continuous on $(a,b) \subseteq \bbR$, while $AC_{loc}([a,b))$
represents absolutely continuous functions on $[a,c]$ for any $a < c < b$. Whenever possible 
we will omit Lebesgue measure $dx$ in $L^p((a,b); dx)$ and simply write $L^p((a,b))$, $p \geq 1$, 
instead. We also abbreviate $\bbN_0 = \bbN \cup \{0\}$. If $T$ is a linear operator mapping 
(a subspace of\,) a Hilbert space into another, $\dom(T)$ denotes the domain and $\ran(T)$ is 
the range of $T$. The Banach space of bounded linear operators on a separable complex 
Hilbert space $\cH$ is denoted by $\cB(\cH)$. 

The spectrum and point spectrum (i.e., the set of eigenvalues) of a closed operator $T$ are 
denoted by $\sigma(T)$ and $\sigma_p(T)$. If $N$ is normal in $\cH$, the 
absolutely and singularly continuous spectrum of $N$ are denoted by $\sigma_{ac}(N)$  
and $\sigma_{sc}(N)$, respectively.

\section{An Integral Inequality} \lb{s2}

The following theorem will be applied repeatedly in the next section to prove
properties of functions in the space $H_{n}([0,\infty)).$ This integral inequality in
$L^{2}((a,b))$ was established by Talenti \cite{Ta69} and Tomaselli
\cite{To69} in 1969. Unaware of their independent proofs, Chisholm and
Everitt \cite{CE70/71} established Theorem \ref{t2.1} in
1971; see also \cite{CEL99} for a more general result in the conjugate index
case $1/p+1/q=1.$ In addition, a 1972 paper by Muckenhoupt \cite{Mu72}
has a result which contains Theorem \ref{t2.1}. For further
information, there is an excellent historical account of Theorem
\ref{t2.1} in the book \cite[Ch.~4, pp.~33--37]{KMP07}.

\begin{theorem} \lb{t2.1}
Let $(a,b) \subseteq \bbR$, $-\infty\leq a<b\leq\infty$, and $w:(a,b) \rightarrow\bbR $ be 
Lebesgue measurable and nonnegative a.e.~on $(a,b)$. In addition, suppose 
$\varphi,\psi :(a,b)\rightarrow\bbC $ are Lebesgue measurable functions satisfying the
following conditions: \\[1mm] 
$(i)$ $\varphi,\psi\in L_{loc}^{2}((a,b);wdx)$. \\[1mm] 
$(ii)$ for some $($and hence for all\,$)$ $c\in(a,b)$, 
\begin{equation}
\varphi\in L^{2}((a,c];wdx),\text{ }\psi\in L^{2}([c,b);wdx). 
\end{equation}
$(iii)$ for all $[\alpha,\beta]\subset(a,b),$ one has
\begin{equation}
\int_{a}^{\alpha}\left\vert \varphi(t)\right\vert ^{2} w(t) dt > 0\text{ and }\int_{\beta
}^{b}\left\vert \psi(t)\right\vert ^{2} w(t) dt > 0.
\end{equation}
Define the linear operators $A, B:L^{2}((a,b);wdx)\rightarrow 
L_{loc}^{2}((a,b);wdx)$ by
\begin{equation}
(Af)(x):=\varphi(x)\int_{x}^{b}\psi(t)f(t)w(t)dt, \quad f\in L_{loc}^{2}((a,b);wdx), 
\end{equation}
and
\begin{equation}
(Bf)(x):=\psi(x)\int_{a}^{x}\varphi(t)f(t)w(t)dt, \quad f\in L_{loc}^{2}((a,b);wdx),
\end{equation}
and the function $K:(a,b)\rightarrow\bbR $ by
\begin{equation}
K(x):=\left(  \int_{a}^{x}\left\vert \varphi(t)\right\vert ^{2}w(t)dt\right)
^{1/2}\left(  \int_{x}^{b}\left\vert \psi(t)\right\vert ^{2}w(t)dt\right)
^{1/2}.
\end{equation}
Then $A$ and $B$ are bounded linear operators in $L^{2}((a,b);wdx)$ if and only if
\begin{equation}
K:=\sup_{x\in (a,b)}K(x)<\infty.
\end{equation}
Moreover, if $K < \infty$, then $A$ and $B$ are adjoints of each other in $L^{2}((a,b);wdx)$, with 
\begin{equation}
\|Af\|_{L^{2}((a,b);wdx)} = \|Bf\| _{L^{2}((a,b);wdx)} 
\leq 2K\|f\| _{L^{2}((a,b);wdx)},  
\quad f\in L^{2}((a,b);wdx),    \lb{2.7}
\end{equation}
in particular, 
\begin{equation}
\| A\|_{\cB(L^2((a,b); wdx))}= \Vert B\|_{\cB(L^2((a,b); wdx))} \leq 2K. \lb{2.8}
\end{equation}
\end{theorem}

\section{The Function Spaces $H_{n}([0,\infty))$ and $\hatt H_{n}((0,\infty))$} 
\lb{s3}

Let $H_{n}([0,\infty))$ and $\hatt H_{n}((0,\infty))$, $n \in \bbN$, be the spaces defined, respectively, in \eqref{1.4} and \eqref{1.7}, that is, 
\begin{align}
\begin{split} 
 H_{n}([0,\infty)):= \big\{f:[0,\infty)\rightarrow\bbC  \, \big| \,  f^{(j)}\in 
AC_{loc}([0,\infty)); \, f^{(n)}\in L^{2}((0,\infty));   \lb{3.0}&  \\
 f^{(j)}(0)=0, \, j=0,1,\dots,n-1&\big\}. 
\end{split} 
\end{align}
and
\begin{align}
& \hatt H_{n}((0,\infty)):= \big\{f:(0,\infty)\rightarrow\bbC \, \big| \, f^{(j)}\in
AC_{loc}((0,\infty)), \, j=0,1,\dots,n-1;  \no \\
& \hspace*{7cm} f^{(n)},f/x^{n}\in L^{2}((0,\infty))\big\}. \lb{3.0a} 
\end{align}

With\ the inner product 
$(\, \cdot \, ,\, \cdot \,)_{H_{n}([0,\infty))}$ as defined in \eqref{1.5}, that is,
\begin{equation}
(f,g)_{H_{n}([0,\infty))}:=\int_{0}^{\infty} \ol{f^{(n)}(x)} \, g^{(n)}(x) \, dx, \quad f,g\in
H_{n}([0,\infty)),     \lb{3.1} 
\end{equation}
one observes that
\begin{equation}
\left\Vert f\right\Vert _{H_{n}([0,\infty))}=\big\Vert f^{(n)}\big\Vert _{L^2((0,\infty))},   \quad 
f\in H_{n}([0,\infty)).     \lb{3.2} 
\end{equation}
Using \eqref{3.2}, we now prove the following result.

\begin{proposition} \lb{p3.1} 
The inner product space $(H_{n}([0,\infty)),(\, \cdot \, , \, \cdot \,)_{H_{n}([0,\infty))})$ is 
actually a Hilbert space. In addition, $C_0^{\infty}((0,\infty))$ is dense in 
$(H_{n}([0,\infty)),(\, \cdot \, , \, \cdot \,)_{H_{n}([0,\infty))})$. 
\end{proposition}
\begin{proof}
First we note that $f \in H_{n}([0,\infty))$ and $\left\Vert f\right\Vert _{H_{n}([0,\infty))}=0$ implies 
$f^{(n)} = 0$ a.e.~on $(0,\infty)$ and hence $f=0$ as $f^{(j)}(0)=0$ for
$j=0,1,\dots,n-1$. 

Next, let $\{f_{m}\}_{m=1}^{\infty}\subset H_{n}([0,\infty))$ be a Cauchy sequence. Then, from
\eqref{3.2}, one infers that $\{f_{m}^{(n)}\}_{m=1}^{\infty}$ is Cauchy
in $L^{2}((0,\infty)).$ Consequently, there exists $g\in L^{2}((0,\infty))$ such
that
\begin{equation}
f_{m}^{(n)} \underset{m \to \infty}{\longrightarrow} g\text{ in }L^{2}((0,\infty)). \lb{3.3}
\end{equation}
Define
\begin{equation}
f(x):=\int_{0}^{x}\int_{0}^{t_{1}}\cdots\int_{0}^{t_{n-1}}g(u)\,du\,dt_{n-1}\dots dt_{1}.
\end{equation}
Noting that $f^{(j)}\in AC_{loc}([0,\infty))$ and $f^{(j)}(0)=0$ for
$j=0,1,\dots,n-1,$ and $f^{(n)}=g$ a.e.~on $(0,\infty)$, one obtains $f\in H_{n}([0,\infty)).$
Furthermore, by \eqref{3.3}, 
\begin{equation}
\left\Vert f_{m}-f\right\Vert _{H_{n}([0,\infty))}=\big\Vert f_{m}^{(n)}-f^{(n)}
\big\Vert_{L^2((0,\infty))}=\big\Vert f_{m}^{(n)}-g\big\Vert _{L^2((0,\infty))} \underset{m \to \infty}{\longrightarrow}
0.
\end{equation}
This completes the proof that $(H_{n}([0,\infty)),(\, \cdot \,,\, \cdot \,)_{H_{n}([0,\infty))})$ is a Hilbert space.

To prove density of $C_0^{\infty}((0,\infty))$ in $H_{n}([0,\infty))$ we assume that 
$g_0 \in H_{n}([0,\infty))$ is perpendicular to $C_0^{\infty}((0,\infty))$ with respect to the inner  
product introduced in \eqref{3.1}. Viewing $\ol{g_{0}}$ as a regular distribution $T_{\ol{g_0}}$ yields 
\begin{equation}
T_{\ol{g_0}}(\varphi):= \int_{0}^{\infty} \ol{g_{0}(x)} \varphi(x) \, dx,   
\quad \varphi \in C_0^{\infty}((0,\infty)).    \lb{3.6} 
\end{equation}
Since $g_0 \perp C_0^{\infty}((0,\infty))$ one concludes that 
\begin{equation}
(g_0, \varphi)_{H_n([0,\infty))}  = (g^{(n)}_0, \varphi^{(n)})_{L^{2}((0,\infty))} = 0, \quad 
\varphi \in C_0^{\infty}((0,\infty)). 
\end{equation}
Since  $g_{0}^{(j)} \in AC_{loc}([0,\infty))$ for $j = 0, \dots, n-1$, one can integrate by parts $n$ times to yield
\begin{align} \lb{3.8}
& \int_{0}^{\infty}  \ol{g_{0}(x)} \varphi^{(2n)}(x)dx 
=  (-1)^{n} \int_{0}^{\infty}  \ol{g_{0}^{(n)}(x)} \varphi^{(n)}(x)dx  \no \\
& \quad =  (-1)^{n}(g^{(n)}_0, \varphi^{(n)})_{L^{2}((0,\infty))} 
= (-1)^{n} (g_0, \varphi)_{H_n([0,\infty))}  = 0, \quad \varphi \in C_0^{\infty}((0,\infty)). 
\end{align}
The left-hand side of \eqref{3.8} is the $2n$\textsuperscript{th}-distributional derivative of 
$\ol{g_0}$. Hence,
\begin{equation}
 T^{(2n)}_{\ol{g_0}}(\varphi) = T_{\ol{g_0}}(\varphi^{(2n)}) 
 = (-1)^{n} (g_0, \varphi)_{H_n([0,\infty))}  = 0,  
\quad  \varphi \in C_0^{\infty}((0,\infty)).
\end{equation}
Thus, by \cite[Thm. 6.11 and Exercise 6.12]{LL01}, it follows that $T_{\ol{g_0}}$, or rather 
$g_0$, is a polynomial of degree at most $2n-1$, 
\begin{equation}
g_{0}(x) = \sum_{ k = 0}^{2n-1} c_{k}x^{k}. 
\end{equation}
However, as $g_{0} \in H_{n}([0,\infty))$, it follows that $g_{0} \equiv 0$.
Indeed, as $g^{(j)}_{0}(0) = 0$ for $j=0, \dots, n-1$ we  have
\begin{equation}
c_{0} = c_{1} = \dots = c_{n-1} = 0
\end{equation}
Furthermore, the condition $g^{(n)}_{0} \in L^{2}((0,\infty))$ yields
\begin{equation}
c_{n} = c_{n+1} = \dots = c_{2n-1} = 0, 
\end{equation}
completing the proof.
\end{proof}

For general statements concerning completeness, see also \cite[p.~31]{Bu98}.

Using Theorem \ref{t2.1}, we now prove the following theorem. The
results of this theorem will be used in our new proof of the Birman
inequalities, defined in \eqref{1.1}, on $H_{n}([0,\infty))$ in the next section.

\begin{theorem} \lb{t3.2}
Let $f\in H_{n}([0,\infty)).$ Then the following items $(i)$--$(iii)$ hold: \\[1mm] 
$(i)$ $f^{(n-j)}/x^{j}\in L^{2}((0,\infty))$, $j=0,1,\dots, n$.  \\[2mm] 
$(ii)$ $\lim_{x\uparrow\infty}\dfrac{|f^{(j)}(x)|^{2}
}{x^{2n-2j-1}}=0$, $j=0,1,\dots,n-1$. \\[2mm] 
$(iii)$ $\lim_{x\downarrow 0}\dfrac{|f^{(j)}(x)|^{2} }{x^{2n-2j-1}}=0$, $j=0,1,\dots,n-1.$ \\[1mm] 
In addition,  
\begin{equation}
f\in \hatt H_{n}((0,\infty)) \, \text{  implies  } \, f^{\prime} \in \hatt H_{n-1}((0,\infty)).  \lb{3.15} 
\end{equation} 
\end{theorem}
\begin{proof}
We may assume without loss of generality that $f\in H_{n}([0,\infty))$ is real-valued. To
prove item $(i)$, one notes that the case $j=0$ is valid by definition of $H_{n}([0,\infty)).$ For
$j=1,$ one uses Theorem \ref{t2.1}. Since $f^{(n-1)}(0)=0,$ one has
the identity
\begin{equation}
\frac{f^{(n-1)}(x)}{x}=\frac{1}{x}\int_{0}^{x}f^{(n)}
(t) \, dt.\lb{3.13}
\end{equation}
Next, one  applies Theorem \ref{t2.1} with $\varphi(x)=1,$ $\psi(x)=1/x\ $and
$w(x)=1.$ Since
\begin{equation}
\int_{0}^{x}1^{2} \, dt \int_{x}^{\infty}\frac{1}{t^{2}} \, dt=1,
\end{equation}
one infers from \eqref{3.13} that $f^{(n-1)}/x$ $\in L^{2}((0,\infty))$ and
this establishes $(i)$ for $j=1$. For $j=2,$ one obtains 
\begin{equation}
\frac{f^{(n-2)}(x)}{x^{2}}=\frac{1}{x^{2}}\int_{0}^{x}t\frac{f^{(n-1)}(t)}
{t} \, dt.
\end{equation}
Again, with $\varphi(x)=x,$ $\psi(x)=1/x^{2}\ $and $w(x)=1$ and noting that
\begin{equation}
\int_{0}^{x}t^{2} \, dt\int_{x}^{\infty}\frac{1}{t^{4}} \, dt=\frac{1}{9},
\end{equation}
one concludes from Theorem \ref{t2.1} that $f^{(n-2)}/x^{2}$ $\in
L^{2}((0,\infty)).$ By induction, for $j=0,1,\dots,n-1,$ one obtains
\begin{equation}
\frac{f^{(n-j-1)}(x)}{x^{j+1}}=\frac{1}{x^{j+1}}\int_{0}^{x}t^{j}
\frac{f^{(n-j)}(t)}{t^{j}} \, dt, 
\end{equation}
assuming $f^{(n-j)}/x^{j}\in L^{2}((0,\infty))$ (and $f^{(n-j-1)}
(0)=0).$ Since
\begin{equation}
\int_{0}^{x}t^{2j} \, dt\int_{x}^{\infty}\frac{dt}{t^{2j+2}}=\frac{1}{\left(
2j+1\right)  ^{2}},
\end{equation}
one obtains from Theorem \ref{t2.1} that $f^{(n-j-1)}/x^{j+1}\in
L^{2}((0,\infty))$, completing the proof of item $(i)$. In particular, 
$f^{(j)}/x^{n-j}$, $f^{(j+1)}/x^{n-j-1}$ $\in L^{2}((0,\infty))$ and 
H\"{o}lder's inequality implies 
\begin{equation}
\frac{f^{(j)}f^{(j+1)}}{x^{2n-2j-1}}\in L^{1}((0,\infty)).\lb{3.19}
\end{equation}
Using integration by parts one obtains that, for any $[a,b]\subset(0,\infty)$ and
$j=0,1,\dots,n-1,$
\begin{align}
\int_{a}^{b} \frac{[f^{(j)}(x)]^2}{x^{2(n-j)}} \, dx &  =-\frac
{1}{2n-2j-1}\int_{a}^{b} [f^{(j)}(x)]^{2}\bigg(  \frac
{1}{x^{2n-2j-1}}\bigg)^{\prime}dx \no \\
&  =-\frac{1}{2n-2j-1}\bigg(\frac{[f^{(j)}(x)]^{2}
}{x^{2n-2j-1}}\bigg\vert _{a}^{b}-2\int_{a}^{b}\frac{f^{(j)}(x)f^{(j+1)}
(x)}{x^{2n-2j-1}} \, dx\bigg).    \lb{3.20}
\end{align}
From part $(i)$ and \eqref{3.19}, both integral terms in the identity in
\eqref{3.20} have finite limits as $a\downarrow 0$ or
$b\uparrow\infty;$ hence both limits
\begin{equation}
\lim_{x\uparrow\infty}\frac{[f^{(j)}(x)]^{2}}{x^{2n-2j-1}
}\text{ and }\lim_{x\downarrow 0}\frac{[f^{(j)}(x)]^{2}
}{x^{2n-2j-1}}, \quad j=0,1,\dots,n-1, 
\end{equation}
exist and are finite. We now establish part $(ii)$. Suppose, to the contrary 
that for some $j\in\{0,1,\dots,n-1\},$
\begin{equation}
\lim_{x\uparrow\infty}\frac{[f^{(j)}(x)]^{2}}{x^{2n-2j-1}
}=c>0.
\end{equation}
Then there exists $X>0$ such that
\begin{equation}
\frac{[f^{(j)}(x)]^{2}}{x^{2n-2j-1}}\geq\frac{c}{2}, \quad x \geq X.
\end{equation}
Multiplying the inequality by $1/x$, integrating and applying item $(i)$ yields
\begin{equation}
\infty>\int_{X}^{\infty} \frac{[f^{(j)}(x)]^2}{x^{2(n-j)}} \, 
dx\geq\frac{c}{2}\int_{X}^{\infty}\frac{1}{x} \, dx=\infty,
\end{equation}
a contradiction. This forces $c=0$ and proves item $(ii)$. A similar argument 
proves part $(iii)$.

The claim \eqref{3.15} is proved as in part $(i)$, choosing $j = n- 1$. 
\end{proof}

\begin{remark} \lb{r3.3}
We emphasize that 
\begin{equation}
H_n([0,\infty)) \neq H^n_0((0,\infty)), \quad n \in \bbN,
\end{equation}
with $H^n_0((0,\infty))$ denoting the standard Sobolev space obtained upon completing  
$C_0^{\infty}((0,\infty))$ in the norm of $H^n((0,\infty))$. (See, however, 
Theorem \ref{t8.2} in the finite interval context.) 

Indeed, $f\in H_{n}([0,\infty))$ does not necessarily
imply that some, or all, of the functions $f,$ $f^{\prime},\dots,f^{(n-1)}$
belong to $L^{2}((0,\infty)).$ In fact, define
\begin{equation}
\widetilde{f}(x)=\left\{
\begin{array}
[c]{ll}
0, & x\text{ near }0,\\
x^{(2n-1)/2}/\ln (x), & x\text{ near }\infty,\text{ }
\end{array}
\right.
\end{equation}
such that
\begin{equation}
\widetilde{f}^{(j)}\in AC_{loc}([0,\infty)), \quad j=0,1,\dots,n.
\end{equation}
Calculations show that $\widetilde{f}\in H_{n}([0,\infty))$, but $\widetilde{f}^{(j)}\notin
L^{2}((0,\infty))$, $j=0,1,\dots,n-1.$ \\
${}$ \hfill $\diamond$
\end{remark}

\medskip

For $n\in\bbN  ,$ let $\hatt H_{n}((0,\infty))$ be as in \eqref{1.7} and pick 
$f\in \hatt H_{n}((0,\infty)).$ Then
\begin{equation}
f^{(n-1)}(1)-f^{(n-1)}(x)=\int_{x}^{1}f^{(n)}(t) \, dt \underset{x\downarrow 0}{\longrightarrow} \int_{0}
^{1}f^{(n)}(t) \, dt;
\end{equation}
hence $f^{(n-1)}(0_{+}) = \lim_{x \downarrow 0} f^{(n-1)}(x)$ exists and is finite. By defining 
$f^{(n-1)}(0):=f^{(n-1)}(0_{+}),$ we see that $f^{(n-1)}\in 
AC_{loc }([0,\infty)).$ A similar argument shows that
\begin{equation}
f\in \hatt H_{n}((0,\infty)) \, \text{ implies } \, f^{(j)}\in 
AC_{loc}([0,\infty)), \; j=0,1,\dots,n-1. \lb{3.29}
\end{equation}
We now prove the following result. 

\begin{theorem} \lb{t3.4} 
For each $n\in\bbN$,  
\begin{equation} 
H_{n}([0,\infty)) = \hatt H_{n}((0,\infty))
\end{equation} 
as sets.
\end{theorem}
\begin{proof}
Let $n \in \bbN$. If $f\in H_{n}([0,\infty)),$ one concludes by Theorem \ref{t3.2} that 
$f/x^{n}\in L^{2}((0,\infty))$ and hence $H_{n}([0,\infty))\subseteq \hatt H_{n}((0,\infty)).$ 
Next, we show that
\begin{equation}
\hatt H_{n}((0,\infty)) \subseteq H_{n}([0,\infty)).      \lb{3.31}
\end{equation}
One notes that
\begin{equation}
f\in \hatt H_{n}((0,\infty)) \, \text{ implies } \, f^{\prime}\in \hatt H_{n-1} ((0,\infty));    \lb{3.32} 
\end{equation}
indeed, this follows from Theorem \ref{t3.2}\,$(i)$ with $j=n-1.$ Repeated application of
\eqref{3.32} yields 
\begin{equation}
f\in \hatt H_{n}((0,\infty)) \, \text{ implies } \, f^{(j)}\in \hatt H_{n-j}((0,\infty)), 
\quad j=0,1,\dots, n-1.      \lb{3.33} 
\end{equation}
Next, we claim that 
\begin{equation}
f\in \hatt H_{n}((0,\infty)) \, \text{ implies } \,  f(0)=0 \, \text{ (cf.\ \eqref{3.29});}   \lb{3.34} 
\end{equation}
to prove \eqref{3.34}, suppose $\left\vert f(0)\right\vert =c>0.$ By
continuity, there exists $\delta>0$ such that 
\begin{equation}
\left\vert f(x)\right\vert \geq c/2 \,  \text{ for all } \, x\in [0,\delta].
\end{equation}
Then 
\begin{equation}
\infty>\int_{0}^{\infty} \dfrac{|f(x)|^2}{x^{2n}} dx 
\geq\int_{0}^{\delta} \dfrac{|f(x)|^2}{x^{2n}} 
dx\geq\dfrac{c^{2}}{4}\int_{0}^{\delta}\dfrac{dx}{x^{2n}}=\infty,
\end{equation}
a contradiction. Hence, $f(0)=0$ proving \eqref{3.34}. Applying
this argument to the implication in \eqref{3.33} yields 
\begin{equation}
f\in \hatt H_{n}((0,\infty)) \, \text{ implies } \, f^{(j)}(0)=0, \quad j=0,1,\dots, n-1, 
\end{equation}
proving \eqref{3.31}. 
\end{proof}

Next, we offer one more characterization of $H_n([0,\infty))$. Define for each $n \in \bbN$, 
\begin{equation}
D_{n}([0,\infty)):=\left\{  \int_{0}^{x}\int_{0}^{t_{1}}\cdots\int_{0}^{t_{n-1} 
}f(t) \, dt \, dt_{n-1}\dots dt_{1}\, \bigg| \, f\in L^{2}((0,\infty))\right\}  .
\end{equation}
In particular, $D_{1}([0,\infty)) = \left\{  \int_{0}^{x}f(t)dt\, \Big| \,  f\in L^{2}((0,\infty))\right\}.$

\begin{theorem} \lb{t3.5} 
For each $n\in\bbN  ,$ $H_{n}([0,\infty)) = D_{n}([0,\infty)).$
\end{theorem}
\begin{proof}
Following the discussion in the proof of Proposition \ref{p3.1}, one concludes that 
$D_{n}([0,\infty)) \subseteq H_{n}([0,\infty))$, and hence 
it suffices to show $H_{n}([0,\infty)) \subseteq D_{n}([0,\infty)).$ To this end it is instructive to
first consider the case $n=1.$ Let $f\in H_{1}([0,\infty))$ so $f^{\prime}\in
L^{2}((0,\infty)).$ By H\"{o}lder's inequality, $f^{\prime}\in L_{loc
}^{1}((0,\infty));$ indeed, for $0\leq x<y<\infty,$
\begin{align}
\begin{split} 
\int_{x}^{y}\left\vert f^{\prime}(t)\right\vert dt &  \leq\left(  \int_{x}
^{y}\left\vert f^{\prime}(t)\right\vert ^{2}dt\right)  ^{1/2} \left(
\int_{x}^{y}1^{2} \, dt\right)  ^{1/2}\\
&  \leq\left(  \int_{0}^{\infty}\left\vert f^{\prime}(t)\right\vert
^{2}dt\right)  ^{1/2} |y-x|^{1/2}<\infty.
\end{split} 
\end{align}
For $x\geq0,$ let $h(x):=\int_{0}^{x}f^{\prime}(t) \, dt.$ Then $h\in
D_{1}([0,\infty)) \cap H_{1}([0,\infty)).$ By standard integration arguments, 
$h=f+C$ on $[0,\infty)$
for some constant $C.$ Since $f(0)=h(0)=0,$ $C=0$ and thus $f=h \in D_{1}([0,\infty)).$ 

In general, let $f\in H_{n}([0,\infty)).$ Then 
$f^{(n)}\in L^{2}((0,\infty))$ and, as above, $f^{(n)}\in L_{loc}^{1}((0,\infty)).$
Define, for $x\geq0,$
\begin{equation} 
h(x):=\int_{0}^{x}\int_{0}^{t_{1}}\cdots\int_{0}^{t_{n-1}}f^{(n)}
(t) \, dt \, dt_{n-1}\dots dt_{1}, 
\end{equation}
so $h\in D_{n}([0,\infty))$ and $h^{(n)}(x)=f^{(n)}(x)$ for $a.e.$ $x\geq0.$ Mimicking the
argument for $n=1,$ it follows that $h=f+p$ on $[0,\infty)$ for some
polynomial $p$ of degree less than or equal to $n-1.$ However, $h^{(j)}(0)=f^{(j)}(0)$ for
$j=0,1,\dots, n-1;$ that is to say, $p^{(j)}(0)=0$ for $j=0,1,\dots, n-1.$
Hence $p\equiv0$ and thus $f=h\in D_{n}([0,\infty)).$
\end{proof}

We conclude this section with the following result which is interesting in its
own right; the proof of part $(i)$ is contained in the proof of Theorem
\ref{t3.4} and the proof of part $(ii)$ follows from Theorem \ref{t3.2}.

\begin{theorem} \lb{t3.10} 
Let $n\in\bbN$.  Suppose $f:[0,\infty)\rightarrow
\bbC $ satisfies $f^{(j)}\in AC_{loc}([0,\infty))$, $j=0,1,\dots,n-1.$ 
Then the following assertions $(i)$ and $(ii)$ hold:\\[1mm] 
$(i)$ If $f/x^{n},f^{(n)}\in L^{2}((0,\infty)),$ then $f^{(j)}(0)=0$, 
$j=0,1,\dots,n-1.$ \\[1mm] 
$(ii)$ If $f^{(n)}\in L^{2}((0,\infty))$ and $f^{(j)}(0)=0$, 
$j=0,1,\dots,n-1,$ then $f/x^{n}\in L^{2}((0,\infty)).$ In fact, $f^{(n-j)}
/x^{j}\in L^{2}((0,\infty))$, $j=0,1,\dots,n.$ 
\end{theorem}

\section{A New Proof of Birman's Sequence of Hardy--Rellich-type Inequalities} 
\lb{s4}

For the sake of completeness, we first give Glazman's proof (see 
\cite[pp.~83--84]{Gl66}) of the Birman inequalities in \eqref{1.1}; actually, we 
provide a slight generalization including a power weight. 
Birman does not give these explicit details in \cite{Bi66}, but it is clear
that he knew this proof. We note that another proof of the inequalities in
\eqref{1.1}, for $f$ $\in C_0^{n}((0,\infty)),$ follows from
repeated applications of Lemmas 5.3.1 and 5.3.3 in Davies' text 
\cite[pp.~104--105]{Da95}. Subsequently, we present a new proof of the 
Birman inequalities whose interest lies in the fact that it essentially consists of 
repeated use of Hardy's inequality.

We start with a slight extension of Glazman's result, \cite[pp.~83--84]{Gl66}: 

\begin{theorem} \lb{t4.1} 
Let $n\in\bbN$, $\alpha \in \bbR$, and $f\in C_0^{n}((0,\infty))$ be real-valued. Then 
\begin{equation}
\int_{0}^{\infty} x^{\alpha} \big[f^{(n)}(x)\big]^{2}dx\geq\frac{\big[\prod_{j=1}^n (2n + 1 - 2j - \alpha)\big]^{2}}{2^{2n}} 
\int_{0}^{\infty} \frac{[f(x)]^2}{x^{2n - \alpha}} \, dx. \lb{4.1}
\end{equation}
Moreover, if $f \not \equiv 0$, the inequalities \eqref{4.1} are strict.
\end{theorem} 
\begin{proof}
Since
\begin{equation}
f(x)^2=2\int_{0}^{x}f(t)f^{\prime}(t) \, dt,     \lb{4.2}
\end{equation}
one infers that 
\begin{align}
\int_{0}^{\infty} \dfrac{f(x)^2}{x^{2n - \alpha}} \, dx  &  =2\int
_{0}^{\infty}x^{\alpha - 2n}\left(  \int_{0}^{x}f(t)f^{\prime}(t) \, dt\right)  dx  \no \\
&  =2\int_{0}^{\infty}f(t)f^{\prime}(t) \bigg(\int_{t}^{\infty}x^{\alpha - 2n} \, dx\bigg) dt   \no \\
&  =\dfrac{2}{2n-1-\alpha}\int_{0}^{\infty}t^{\alpha + 1-2n}f(t)f^{\prime}(t) \, dt    \no \\
&  \leq\dfrac{2}{2n-1-\alpha}\left(  \int_{0}^{\infty} \dfrac{[f(x)]^2}{x^{2n-\alpha}} \, dx\right)^{1/2}
\left(  \int_{0}^{\infty} \dfrac
{[f^{\prime}(x)]^2}{x^{2(n-1) - \alpha}} \, dx\right)^{1/2}.    \lb{4.3} 
\end{align}
Here we used the elementary fact
\begin{equation}
\int_0^y f_1(x) \bigg(\int_0^x f_2(t) \, dt\bigg) dx 
= \int_0^y f_2(t) \bigg(\int_t^y f_1(x) \, dx \bigg) dt, \quad y > 0   \lb{4.4} 
\end{equation} 
(verified, e.g., by differentiating with respect to $y$, assuming appropriate integrability conditions 
on $f_1, f_2$), in the second line of \eqref{4.3}, and employed Cauchy--Schwarz in the final step of 
\eqref{4.3}. This implies 
\begin{equation}
\int_{0}^{\infty} \dfrac{[f(x)]^2}{x^{2n-\alpha}} \, dx\leq\left(
\dfrac{2}{2n-1 - \alpha}\right)  ^{2}\int_{0}^{\infty} \dfrac{[f^{\prime}(x)]^2}{x^{2(n-1)-\alpha}} \, dx. 
\lb{4.4a} 
\end{equation}
By iteration, for $\alpha \neq 2n+1 - 2j$, $j=0,1,\dots, n,$ one obtains 
\begin{align}
\begin{split}
& \dfrac{((2n-1-\alpha)(2n-3-\alpha)\cdots(2n+1-2j-\alpha))^{2}}{2^{2j}} 
\int_{0}^{\infty} \dfrac{[f(x)]^2}{x^{2n-\alpha}} \, dx   \\
& \quad \leq
\int_{0}^{\infty} \dfrac{[f^{(j)}(x)]^2}{x^{2(n-j)-\alpha}} \, dx. \lb{4.5}
\end{split} 
\end{align}
Letting $j=n$ in \eqref{4.5} implies \eqref{4.1} on $C_0^{n}((0,\infty))$. 
Inequality \eqref{4.4a} becomes trivial if $\alpha = 2n + 1 - 2j$, $1 \leq j \leq n$.

To prove that all inequalities are strict unless $f \equiv 0$, one just has to check the case of 
equality in all the Cauchy inequalities involved. The latter are of the type
\begin{align}
\begin{split} 
& \int_0^{\infty} \frac{f^{(j-1)}(x)}{x^{n - (j-1) - (\alpha/2)}}  \frac{f^{(j)}(x)}{x^{n - j - (\alpha/2)}} \, dx \\
& \quad 
\leq \bigg(\int_0^{\infty} \frac{\big[f^{(j-1)}(x)\big]^2dx}{x^{2n - 2(j-1) - \alpha}}\bigg)^{1/2} 
\bigg(\int_0^{\infty} \frac{\big[f^{(j)}(x)\big]^2dx}{x^{2n - 2j - \alpha}}\bigg)^{1/2}, 
\quad 1 \leq j \leq n.     \lb{4.5a} 
\end{split} 
\end{align}
Thus, equality in \eqref{4.5a} holds if and only if there exists some $\beta^2 \in [0,\infty)$ such that 
a.e.~on $(0,\infty)$, 
\begin{equation}
\pm \beta f^{(j-1)}(x) =  x f^{(j)}(x),  \quad 1 \leq j \leq n, 
\end{equation}
with general solution of the form
\begin{equation}
f_{j}(x) = c_{j-1}x^{\pm \beta + j-1} + c_{j-2}x^{j-2} +  c_{j-3}x^{j-3} + \dots +  c_{1}x + c_{0}, \quad 1 \leq j \leq n.  \lb{4.5c}
\end{equation}
The right-hand side in \eqref{4.5c} is not compactly supported, completing the proof. 
\end{proof}

\begin{remark} \lb{r4.2} 
If $f=f_{1}+if_{2}\in H_{n}([0,\infty)),$ where $f_{1}$ and $f_{2}$
are, respectively, the real and imaginary parts of $f,$ it is clear by
definition of $H_{n}([0,\infty))$ that $f_{1},f_{2}\in H_{n}([0,\infty)).$ Moreover, if $f_{1}$ and
$f_{2}$ each satisfy \eqref{1.1}, then $f$ also satisfies
\eqref{1.1}. Indeed,
\begin{align}
\int_{0}^{\infty} \big|f^{(n)}(x)\big|^{2}dx  &  =\int_{0}
^{\infty}\big| f_{1}^{(n)}(x)+if_{2}^{(n)}(x)\big|^{2} \, dx   \no \\
&  =\int_{0}^{\infty}\big(  f_{1}^{(n)}(x)+if_{2}^{(n)}(x)\big)  \big(
f_{1}^{(n)}(x)-if_{2}^{(n)}(x)\big) \, dx   \no \\
&  =\int_{0}^{\infty} \big[f_{1}^{(n)}(x)\big]^{2}dx 
+ \int_{0}^{\infty} \big[f_{2}^{(n)}(x)\big]^{2}dx     \no \\
&  \geq\dfrac{((2n-1)!)^{2}}{2^{2n}}\int_{0}^{\infty}\bigg[ 
\dfrac{f_{1}(x)^2}{x^{2n}} + \dfrac{f_{2}(x)^2}{x^{2n}}\bigg] \, dx   \no \\
&  =\dfrac{((2n-1)!)^{2}}{2^{2n}}\int_{0}^{\infty} \dfrac{|f(x)|^2}{x^{2n}} \, dx.
\end{align}
Consequently, to prove that an arbitrary $f\in H_{n}([0,\infty))$ satisfies the inequality
in \eqref{1.1}, it suffices to assume that $f$ is real-valued. The same argument applies 
of course to inequality \eqref{4.1}.
\hfill $\diamond$
\end{remark}

A closer inspection of Glazman's proof readily reveals that it can be extended to the space
$H_{n}([0,\infty))$:

\begin{corollary} \lb{c4.3}
Let $n\in\bbN$ and $f\in H_{n}([0,\infty)).$  Then, 
\begin{equation}
\int_{0}^{\infty}\big| f^{(n)}(x)\big|^{2}dx\geq\frac{[
(2n-1)!!]^{2}}{2^{2n}}\int_{0}^{\infty} \frac{|f(x)|^2}{x^{2n}} \, dx. \lb{4.6a}
\end{equation}
Moreover, if $f \not \equiv 0$, the inequalities \eqref{4.6a} are strict.
\end{corollary}
\begin{proof}
By Remark \ref{r4.2} it suffices to consider real-valued $f$. Comparing \eqref{4.3}, \eqref{4.5}, and 
\eqref{4.5a} with Theorem \ref{t3.2}\,$(i)$ legitimizes all steps in \eqref{4.2}--\eqref{4.5} (taking 
$\alpha = 0$) for $f\in H_{n}([0,\infty)).$ Finally, also strict inequality holds for 
$H_{n}([0,\infty)) \ni f \not \equiv 0$ as the powers in \eqref{4.5c} do not lie in $H_{n}([0,\infty))$. 
\end{proof}

In our new proof of Birman's inequalities \eqref{1.1} below, we
make repeated use of the elementary inequality
\begin{equation}
2xy\leq\varepsilon x^{2}+\frac{1}{\varepsilon}y^{2}, \quad 
x,y\in\bbR , \; \varepsilon>0, \lb{4.7}
\end{equation}
following instantly from $\big(\varepsilon^{1/2} x- \varepsilon^{-1/2} y\big)^{2}\geq0.$ In \cite{Sc72}, Schmincke established 
various one-parameter integral inequalities using this fact; see also
\cite{GL18} where\ new two-parameter inequalities are given. 

We note that the proof of Theorem \ref{t4.4} below is not shorter than other existing proofs, but we find it interesting as it reduces the sequence of Birman inequalities to repeated use of just the first such inequality, namely, Hardy's inequality (i.e., the case $n=1$ in \eqref{4.9}):

\begin{theorem} \lb{t4.4} 
Let $n\in\bbN$ and $f\in H_{n}([0,\infty))$. Then, 
\begin{equation}
\int_{0}^{\infty}\big| f^{(n)}(x)\big|^{2}dx\geq\frac{[
(2n-1)!!]^{2}}{2^{2n}}\int_{0}^{\infty} \frac{|f(x)|^2}{x^{2n}} \, dx. \lb{4.9}
\end{equation}
Moreover, if $f \not \equiv 0$, the inequalities \eqref{4.9} are strict.
\end{theorem}
\begin{proof}
Let $\varepsilon>0,$ and $f\in H_{n}([0,\infty)),$ $n \in \bbN$. We first prove  
\begin{equation}
\int_{0}^{\infty}\big|f^{(n)}(x)\big|^{2}dx\geq\left\{
\begin{array}
[c]{ll}
(-\varepsilon^{2}+\varepsilon)\int_{0}^{\infty} \frac{|f(x)|^2}{x^2} \, dx, & n=1, \\
\frac{(2n-3)!!}{2^{2n-2}}(-\varepsilon^{2}+(2n-1)\varepsilon)\int_{0}^{\infty} 
\frac{|f(x)|^2}{x^{2n}} \, dx, & n \geq 2.
\end{array}
\right.  \lb{4.8}
\end{equation}
Maximizing over $\varepsilon\in(0,\infty)$ then yields \eqref{4.9}. 

We prove \eqref{4.8} by induction on
$n\in\bbN  .$ For $n=1,$ let $f\in H_{1}$ be real-valued on $[0,\infty);$
see Remark \ref{r4.2}. Then
\begin{align}
\int_{0}^{\infty} \frac{f(x)^2}{x^2} \, dx  &  
=-\int_{0}^{\infty}f(x)^2\left(  \frac{1}{x}\right)  ^{\prime}dx=-\left.  \frac{f(x)^2}
{x}\right\vert _{0}^{\infty}+2\int_{0}^{\infty}\frac{f(x)f^{\prime}(x)}{x} \, dx   \no \\
&  =2\int_{0}^{\infty}\frac{f(x)f^{\prime}(x)}{x}dx\text{ by Theorem
\ref{t3.2}}\,(ii)\text{ and }(iii)    \no \\
&  \leq 2\left(  \int_{0}^{\infty} \frac{f(x)^2}{x^2} \, dx\right)
^{1/2}\left(  \int_{0}^{\infty} [f^{\prime}(x)]^2 \, dx\right)
^{1/2}   \no \\
&  \leq\varepsilon\int_{0}^{\infty} \frac{f(x)^2}{x^2} \, 
dx+\frac{1}{\varepsilon}\int_{0}^{\infty} [f^{\prime}(x)]^{2} \, dx\text{ using }(\text{\ref{4.7}}).
\end{align}
This last inequality can be rewritten as
\begin{equation}
\int_{0}^{\infty} [f^{\prime}(x)]^{2} \, dx\geq(-\varepsilon
^{2}+\varepsilon)\int_{0}^{\infty} \frac{f(x)^2}{x^2} \, dx.
\end{equation}
Since the maximum of $\varepsilon\rightarrow-\varepsilon^{2}+\varepsilon$
occurs at $\varepsilon=1/2$ with maximum value $1/4,$ one concludes that
\begin{equation}
\int_{0}^{\infty} [f^{\prime}(x)]^{2} \, dx\geq\frac{1}{4}\int
_{0}^{\infty} \frac{f(x)^2}{x^2} \, dx\geq(-\varepsilon
^{2}+\varepsilon)\int_{0}^{\infty} \frac{f(x)^2}{x^2} \, dx.
\lb{4.12}
\end{equation}
The inequalities in \eqref{4.12} establish both \eqref{4.9} 
and \eqref{4.8} when $n=1.$
Incidentally, this argument also provides a proof of the classical Hardy
inequality \eqref{1.2}. We now assume that \eqref{4.9} holds for $n = 1,.... k-1$ 
for some $k \in \bbN$.  
Let $f\in H_{k}([0,\infty));$ by
\eqref{1.8}, $f^{\prime}\in H_{k-1}([0,\infty))$ and so, from our induction
hypothesis,
\begin{equation}
\int_{0}^{\infty}\big[f^{(k)}(x)\big]^{2}dx 
= \int_{0}^{\infty}\big[[f'(x)]^{(k-1)}\big]^{2}dx 
\geq\frac{[(2k-3)!!]^{2}}{2^{2k-2}}\int_{0}^{\infty} 
\frac{[f^{\prime}(x)]^2}{x^{2(k-1)}} \, dx. \lb{4.13}
\end{equation}
On the other hand, assuming $f$ is real-valued, we note, from the definition
of $H_{k}([0,\infty))$ and Theorem \ref{t3.2}\,$(i)$ that both $f/x^{k}$ and
$f^{\prime}/x^{k-1}$ belong to $L^{2}((0,\infty)).$ Moreover,
\begin{align}
\int_{0}^{\infty} \frac{f(x)^2}{x^{2k}} \, dx  &  =-\frac
{1}{2k-1}\int_{0}^{\infty}f(x)^2\left(  x^{-2k+1}\right)  ^{\prime
}dx \no \\
&  =-\frac{1}{2k-1}\left(  \left.  \frac{f(x)^2}{x^{2k-1}}\right\vert
_{0}^{\infty}-2\int_{0}^{\infty}\frac{f(x)f^{\prime}(x)}{x^{2k-1}} \, dx\right)
\text{ } \no \\
&  =\frac{2}{2k-1}\int_{0}^{\infty}\frac{f(x)f^{\prime}(x)}{x^{2k-1}}dx\text{
by Theorem \ref{t3.2}}\,(ii)\text{ and } (iii) \lb{4.14} \\
&  \leq\frac{2}{2k-1}\left(  \int_{0}^{\infty} \frac{f(x)^2}{x^{2k}} \, dx\right)^{1/2}
\left(  \int_{0}^{\infty} \frac
{[f^{\prime}(x)]^2}{x^{2(k-1)}} \, dx\right)  ^{1/2}  \no \\
&  \leq\frac{1}{2k-1}\left(  \varepsilon\int_{0}^{\infty} \frac
{f(x)^2}{x^{2k}} \, dx+\frac{1}{\varepsilon}\int_{0}^{\infty} 
\frac{[f^{\prime}(x)]^2}{x^{2(k-1)}} \, dx\right)  \text{ by }
(\text{\ref{4.7}}). \no 
\end{align}
Rearranging terms in this last inequality yields 
\begin{equation}
\int_{0}^{\infty} \frac{[f^{\prime}(x)]^2}{x^{2(k-1)}} \, 
dx\geq\left(  -\varepsilon^{2}+(2k-1)\varepsilon\right)  \int_{0}^{\infty} 
\frac{f(x)^2}{x^{2k}} \, dx. \lb{4.15}
\end{equation}
Combining \eqref{4.13} and \eqref{4.15}, one obtains
\begin{equation}
\int_{0}^{\infty}\big[f^{(k)}(x)\big]^{2} \, dx\geq\frac{[
(2k-3)!!]^{2}}{2^{2k-2}}(-\varepsilon^{2}+(2k-1)\varepsilon)\int
_{0}^{\infty} \frac{f(x)^2}{x^{2k}} \, dx,   \lb{4.16}
\end{equation}
implying \eqref{4.8}. The maximum of the function
$\varepsilon\rightarrow-\varepsilon^{2}+(2k-1)\varepsilon$ over $(0,\infty)$
occurs at $\varepsilon=(2k-1)/2$ with the maximum value being $(2k-1)^{2}/4.$
Substituting this value into \eqref{4.16} yields
\begin{align}
\begin{split} 
\int_{0}^{\infty}\big[f^{(k)}(x)\big]^{2} \, dx  &  \geq\frac{[
(2k-3)!!]^{2}}{2^{2k-2}}\frac{(2k-1)^{2}}{4}\int_{0}^{\infty} 
\frac{f(x)^2}{x^{2k}} \, dx   \\
&  =\frac{[(2k-1)!!]^{2}}{2^{2k}}\int_{0}^{\infty} 
\frac{f(x)^2}{x^{2k}} \, dx,
\end{split} 
\end{align}
completing the proof of \eqref{4.9}. Strict inequality in \eqref{4.9} is clear from Corollary \ref{c4.3}, alternatively, one can apply the argument following \eqref{4.5a} (with $j=1$, $\alpha = 0$) and a similar one involving the final $\varepsilon$-step in \eqref{4.14}. 
\end{proof}

\begin{remark} \lb{r4.5} 
Hardy's work on his celebrated inequality started in 1915, \cite{Ha15} (see also 
\cite{Ha19}--\cite{Ha25}, \cite[Sect.~9.8]{HLP88}, and the historical comments in 
\cite[Chs.~1, 3, App.]{KMP07}). Higher-order Hardy inequalities, including weight functions, are discussed in \cite[Ch.~4]{KP03} and \cite[Sect.~10]{OK90}, however, Birman's sequence of inequalities, \cite{Bi66}, is not mentioned in these sources. 
\hfill $\diamond$
\end{remark}

\begin{remark} \lb{r4.6} 
The characterization of functions in $H_{n}$ in Theorem \ref{t3.5} provides
us with two equivalent ways of expressing Birman's inequalities. Indeed, we have
already established that 
\begin{equation} 
\int_{0}^{\infty} \big| f^{(n)}(x) \big|^{2}dx\geq\dfrac{[(2n-1)!!]^{2}}{2^{2n}}
\int_{0}^{\infty} \dfrac{|f(x)|^2}{x^{2n}} \, dx, \quad f\in H_{n}([0,\infty)).
\end{equation}
Alternatively, via Theorem \ref{t3.5}, one can
now express this inequality as
\begin{align} 
& \int_{0}^{\infty}\left\vert f(x)\right\vert ^{2}dx\geq\dfrac{[(2n-1)!!]^{2}}{2^{2n}}\int_{0}^{\infty}\dfrac{1}{x^{2n}}\left\vert
\int_{0}^{x}\int_{0}^{t_{1}}\cdots\int_{0}^{t_{n-1}}f(t) \, dt \, dt_{n-1}\dots
dt_{1}\right\vert ^{2}dx    \no \\ 
& \hspace*{8.5cm} f\in L^{2}((0,\infty)). 
\end{align}
In the case $n=1,$ both forms of Hardy's inequality are
given in \cite{HLP88} (cf.\ Theorem 253, p.\ 175, and Theorem 327, p.\ 240). 
\hfill $\diamond$
\end{remark} 

The constants $[(2n-1)!!]^{2}/2^{2n}$ in Birman's sequence of inequalities \eqref{4.9} are optimal as shown in the following section.

\section{Optimality of Constants} 
\lb{s6}

The principal purpose of this section is to prove the following result:

\begin{theorem} \lb{t6.1} 
The constants $[(2n-1)!!]^{2}/2^{2n}$ in the Birman sequence of Hardy--Rellich-type inequalities 
\begin{equation}
\int_{0}^{\infty}\big| f^{(n)}(x)\big|^{2} \, dx \geq \frac{[(2n-1)!!]^{2}}{2^{2n}}
\int_{0}^{\infty} \frac{|f(x)|^2}{x^{2n}} \, dx, \quad f \in H_{n}([0,\infty)), \; n\in\bbN, \lb{6.4} 
\end{equation}
are optimal in the following sense: The inequality \eqref{6.4} ceases to be valid if  $[(2n-1)!!]^{2}/2^{2n}$ is 
replaced by  $[(2n-1)!!]^{2}/2^{2n} + \varepsilon$ for any $\varepsilon > 0$ on the right-hand side of \eqref{6.4}. 
\end{theorem}
\begin{proof}
We follow the strategy of proof in \cite[p.~4]{BEL15} in connection with weighted Hardy 
inequalities. Fix some $a>0$ and let $\chi_{(0,a)}$ denote the characteristic function on $(0,a)$.
For $\sigma  > -1/2$, let $f_{\sigma } \in L^{2}((0,\infty))$ be given by
\begin{equation}  \lb{6.5}
f_{\sigma }(x) := x^{\sigma } \chi_{(0,a)}(x), \quad x \in (0, \infty), 
\end{equation}
and define
\begin{equation}  \lb{6.6}
F_{n,\sigma }(x) := \int_{0}^{x}  \int_{0}^{t_1} \cdots \int_{0}^{t_{n-1}} f_{\sigma }(u) 
\, du \, dt_{n-1} \dots dt_1, \quad x \in (0, \infty). 
\end{equation}
One observes that $F_{n, \sigma } \in D_{n}([0,\infty)) = H_{n}([0,\infty))$ by Theorem \ref{t3.5}.
Since $F_{n, \sigma }^{(n)} = f_{\sigma }$ a.e. on $[0,\infty)$ one has 
\begin{equation}
\int_{0}^{\infty} \left| F_{n,\sigma }^{(n)}(x) \right|^{2} \, dx
= \int_{0}^{\infty} |f_{\sigma }(x)|^{2} \, dx
= \int_{0}^{a} x^{2\sigma }dx
= \frac{a^{1 + 2\sigma}}{1 + 2\sigma}. 
\end{equation} 
An induction argument then shows that
\begin{equation}
F_{n, \sigma }(x) = 
\begin{cases}
\frac{x^{n + \sigma}}{\prod_{j=1}^{n}(j + \sigma)},  & 0 \leq x \leq a, \\[2mm] 
 \sum_{k=0}^{n-1}b_{k}x^{k},   & x>a,
\end{cases}      \lb{6.7} 
\end{equation}
where
\begin{equation}   \lb{6.9}
b_{k} := \frac{(-1)^{n-k+1}a^{n - k + \sigma}}{k!(n-k-1)!(n - k + \sigma)}, \quad 0 \leq k \leq n-1. 
\end{equation}
A straightforward computation yields
\begin{align}
\int_{0}^{\infty} x^{-2n}F_{n, \sigma }^{2}(x) \, dx 
&= \int_{0}^{a} x^{-2n} F_{n, \sigma }^{2}(x) \, dx 
+ \int_{a}^{\infty} x^{-2n} F_{n,\sigma }^{2}(x) \, dx   \no \\
&= \frac{1}{\prod_{j=1}^{n}(j + \sigma)^2} \int_{0}^{a} x^{2\sigma } \, dx  + \int_{a}^{\infty} x^{-2n}\left( \sum_{k=0}^{n-1} b_{k}x^{k}  \right)^{2} \, dx   \no \\
&= \frac{a^{1 + 2\sigma}}{(1 + 2\sigma)\prod_{j=1}^{n}(j + \sigma)^2} +  C(a), 
\end{align}
where
\begin{equation}
0 < C(a) := \int_{a}^{\infty} x^{-2n}\left( \sum_{k=0}^{n-1} b_{k}x^{k}  \right)^{2} \, dx 
\underset{a \uparrow \infty}{\longrightarrow} 0.
\end{equation}
Thus, 
\begin{align}
\begin{split} 
\frac{ \int_{0}^{\infty} f_{\sigma }^{2}(x) \, dx}{\int_{0}^{\infty} x^{-2n} F_{n, \sigma }^{2}(x) \, dx }
& = \frac{\prod_{j=1}^{n}(j + \sigma)^2}{1 + (1 + 2 \sigma) C(a) a^{-1 - 2 \sigma} \prod_{j=1}^{n}(j + \sigma)^2}     \\
& = \f{\prod_{j=1}^{n}(2 j - 1)^2}{2^{2n}} + (1 + 2 \sigma) D(a, \sigma) + \Oh\big((1 + 2 \sigma)^2\big), 
\end{split} 
\end{align}
where $D(\, \cdot, \cdot)$ satisfies $D(a, \sigma) > 0$, $D(a, -1/2) > 0$, if $a$ is chosen sufficiently large, due to the fact that 
$C(a) \underset{a \uparrow \infty}{\longrightarrow} 0$. Thus, choosing $ \sigma$ sufficiently close to $-1/2$ 
will undercut any choice of $\varepsilon > 0$ in a replacement of $[(2n-1)!!]^{2}/2^{2n}$ by 
$[(2n-1)!!]^{2}/2^{2n} + \varepsilon$ for {\it any} $\varepsilon > 0$ on the right-hand side of \eqref{6.4}.  
\end{proof}

Without going into further details we note that the argument just presented also works for the 
weighted extension of the Birman inequalities \eqref{4.1}. 

We also note that the constants in \eqref{6.4} coincide of course with the ones obtained by Yafaev \cite{Ya99} upon specializing his result to the spherically symmetric case. 

\begin{remark} \lb{r6.2} 
Let $n \in \bbN$. To motivate the choice of the function $f_{\sigma}$, and hence that of 
$F_{n, \sigma}(x) = c \, x^{n +\sigma} = c \, x^{n - (1/2) + \varepsilon}$, 
writing $\sigma = - (1/2) + \varepsilon$, $\varepsilon > 0$, near $x = 0$ in the above proof,  
it suffices to recall that Birman's inequalities, 
\begin{equation} 
\int_{0}^{\infty} \big| f^{(n)}(x) \big|^{2}dx \geq \dfrac{[(2n-1)!!]^{2}}{2^{2n}}
\int_{0}^{\infty} \dfrac{|f(x)|^2}{x^{2n}} \, dx, \quad f\in H_{n}([0,\infty)),
\end{equation}
are naturally associated with the differential expression 
\begin{equation}
\tau_{2n} := (-1)^n \f{d^{2n}}{dx^{2n}} - \dfrac{[(2n-1)!!]^{2}}{2^{2n}} \f{1}{|x|^{2n}}, 
\quad x \in (0,\infty).
\end{equation}
According to Birman's inequalities, 
\begin{equation}
\tau_{2n}\big|_{C_0^{\infty}((0,\infty))} \geq 0, 
\end{equation}
and the function $y_n(x)=x^{n-(1/2)}$, $x \geq 0$, satisfies 
\begin{equation}
\tau_{2n} y_n = 0
\end{equation}
and hence formally saturates the lower bound $0$ of $\tau_{2n}$. To ensure membership in 
$H_n([0,\infty))$ one thus regularizes $y_n$ with the help of the parameter $\varepsilon > 0$, 
yielding  $F_{n, \sigma}(x)= c \, x^{n - (1/2) + \varepsilon}$, $x \geq 0$.  \hfill $\diamond$
\end{remark}

\section{The Continuous Ces\`aro Operator $T_1$ and its Generalizations $T_n$} 
\lb{s7}

As shown in the proof of Proposition \ref{p3.1}, for any $f\in
L^{2}((0,\infty)),$ 
\begin{equation}
\int_{0}^{x}\int_{0}^{t_{1}}\cdots\int_{0}^{t_{n-1}}f(u) \, du \, dt_{n-1}\dots
dt_{1}\in H_{n}([0,\infty)),    \lb{7.1}
\end{equation}
thus, we can introduce for $n \in \bbN$, 
\begin{align}
\begin{split} 
& (T_{n}f)(x):=\dfrac{1}{x^{n}}\int_{0}^{x}\int_{0}^{t_{1}}\cdots\int
_{0}^{t_{n-1}}f(u) \, du \, dt_{n-1}\dots dt_{1}, \quad 
\quad  x \in (0,\infty), \\ 
& f \in \dom(T_n) = L^{2}((0,\infty)).    \lb{7.2} 
\end{split} 
\end{align}
The operator $T_{n}$ is patterned after the continuous Ces\`aro operator, 
\begin{equation}
(T_{1}f)(x):=\dfrac{1}{x}\int_{0}^{x}f(t) \, dt, \quad x \in (0,\infty), \; f \in L^{2}((0,\infty)).  \lb{7.3} 
\end{equation}

We now prove the following result.

\begin{theorem} \lb{t7.1} 
Let $n\in\bbN  $ and define $T_{n}$ as in 
\eqref{7.2}, \eqref{7.3}. Then $T_{n}$ is a bounded linear operator on $L^{2}((0,\infty))$
with norm
\begin{equation}
\Vert T_{n}\Vert = \dfrac{2^{n}}{(2n-1)!!}.    \lb{7.4}
\end{equation}
\end{theorem}
\begin{proof}
We abbreviate the reciprocal of the square root of the Birman constant by 
\begin{equation}
B_{n}:=\dfrac{2^{n}}{(2n-1)!!}. \lb{7.5}
\end{equation}
Let $f\in L^{2}((0,\infty))$, and write
\begin{equation}
F(x)=\int_{0}^{x}\int_{0}^{t_{1}}\cdots\int_{0}^{t_{n-1}}f(u) \, du \, dt_{n-1}\dots
dt_{1}, \quad x \in (0, \infty).
\end{equation}
Then $F\in H_{n}([0,\infty))$ and
\begin{equation}
F^{(n)}=f\text{ a.e. on $[0,\infty).$}
\end{equation}
Hence, by \eqref{4.9} in Theorem \ref{t4.4}, one
concludes that
\begin{equation}
\int_{0}^{\infty} \frac{|F(x)|^2}{x^{2n}} \, dx\leq B_{n}
^{2}\int_{0}^{\infty} \big|F^{(n)}(x)\big|^{2}dx=B_{n}^{2}\int
_{0}^{\infty}\left\vert f(x)\right\vert ^{2}dx.\lb{7.8}
\end{equation}
Since $T_{n}f=F/x^{n}$, \eqref{7.8} implies
\begin{equation}
\Vert T_{n}f\Vert_{L^2((0,\infty))} \leq B_{n}\Vert f\Vert_{L^2((0,\infty))};
\end{equation}
in particular, $T_{n}$ is bounded and $\Vert T_{n}\Vert\leq B_{n}$. To show
$\Vert T_{n}\Vert=B_{n},$ let $0<K<B_{n}$ so $K^{2}<B_{n}^{2}$. Since, by
Theorem \ref{t6.1}, the constant $B_{n}^{2}$ is sharp, there
exists $G\in H_{n}([0,\infty))$ such that
\begin{equation}
\int_{0}^{\infty} \frac{|G(x)|^2}{x^{2n}} \, dx > K^{2}\int
_{0}^{\infty} \big|G^{(n)}(x)\big|^{2} \, dx.
\end{equation}
Let $g:=G^{(n)}\in L^{2}((0,\infty))$ such that 
\begin{equation}
G(x)=\int_{0}^{x}\int_{0}^{t_{1}}\cdots\int_{0}^{t_{n-1}}g(u) \, du \, dt_{n-1}\dots
dt_{1}.
\end{equation}
Then $T_{n}g=G/x^{n}$ and
\begin{equation}
\Vert T_{n}g\Vert_{L^2((0,\infty))}>K\Vert g\Vert_{L^2((0,\infty))}.  
\end{equation}
Thus, 
\begin{align}
B_{n} &=\inf \big\{C>0\, \big| \,\Vert T_{n}f\Vert_{L^2((0,\infty))} \leq C\Vert f\Vert_{L^2((0,\infty))} 
\, \text{for all} \, f\in L^{2}((0,\infty))\big\backslash \{0\}\big\}   \no \\ 
&=\Vert T_{n}\Vert, 
\end{align}
completing the proof. 
\end{proof}

It will soon be clear that while $T_n$ is bounded, it is noncompact, see \eqref{7.70}.

\smallskip 

Next, we turn to the inverse of $T_n$ and state the following fact.

\begin{lemma} \lb{l7.2} 
Let $n \in \bbN$, then
\begin{align} \lb{7.14} 
&(T_{n}^{-1}f)(x) = \frac{d^n}{dx^n} x^{n}f(x),  \quad x \in (0,\infty),     \\
& f \in \dom\big(T_{n}^{-1}\big) = \big\{g \in L^{2}((0,\infty)) \ \big|\  g \in AC^{(n-1)}_{loc}((0,\infty)); 
\, (x^{n} g)^{(n)} \in L^{2}((0,\infty));    \no \\
&\hspace{6.8cm} \lim_{x \downarrow 0}(x^{n} g(x))^{(j)}= 0, \, j=0,\dots, n-1 \}.   \no
\end{align}
\end{lemma}
\begin{proof}
For $f \in\ran(T_{n})\cap AC_{loc}^{(n-1)}((0,\infty)),$ consider the equation 
\begin{equation} 
(T_{n}g)(x)=f(x), 
\end{equation} 
or, equivalently, 
\begin{equation}
g(x)=(x^{n} f(x))^{(n)}.     \lb{7.16}
\end{equation}
Since $T_{n}$ is one-to-one, it is clear from \eqref{7.16} that the form 
of the inverse of $T_{n}$ is given by
\begin{equation}
(R_{n}g)(x)=(x^{n} g(x))^{(n)}, \quad x \in (0,\infty).      \lb{7.17}
\end{equation}
We now seek to find the (largest) domain, $\dom(R_{n})\subseteq L^{2}((0,\infty)).$ For any 
such choice of domain, 
\begin{equation}
\dom(R_n) \subseteq \big\{g \in L^{2}((0,\infty)) \, \big| \, g \in AC_{loc}^{(n-1)}((0,\infty)); 
\, (x^{n}g(x))^{(n)} \in L^{2}((0,\infty))\},
\end{equation}
and it is clear that
\begin{equation}
(R_{n}\circ T_{n}) g(x)=R_{n}(T_{n} g)(x)=g(x).      \lb{7.19}
\end{equation}
Conversely, we see that if
\begin{align}
\begin{split} 
g(x)=(T_{n}\circ R_{n}) g(x) &  =T_{n}(R_{n} g)(x)       \\
&  =\dfrac{1}{x^n}\int_{0}^{x}\int_{0}^{t_{1}}\cdots\int_{0}^{t_{n-1}}
(t^{n} g(t))^{(n)}dtdt_{n-1}\cdots dt_{1},        \lb{7.20}
\end{split} 
\end{align}
then it is necessary, for $j=0,1,\ldots,n-1,$ that the limits
\begin{equation}
\lim_{x\rightarrow0^{+}}(x^{n}g(x))^{(j)}
\end{equation}
all must exist and equal $0.$ Consequently, if we define
\begin{align}
& \dom(R_{n}):= \big\{g \in L^{2}((0,\infty)) \, \big| \, g \in AC_{loc}^{(n-1)}((0,\infty)); \, 
(x^{n} g(x))^{(n)}\in L^{2}(0,\infty),     \no \\ 
& \hspace*{2,5cm} \lim_{x\rightarrow
0^{+}}(x^{n} g(x))^{(j)} \, \text{exists and equals $0$, $j =0,1,\ldots n-1$}\},     \lb{7.22}
\end{align}
then, by \eqref{7.19} and \eqref{7.20}, it follows that operator $R_{n},$ defined by 
\eqref{7.17} and \eqref{7.22}, is the inverse of $T_{n}$ in $L^{2}((0,\infty)).$
\end{proof}

One notes that $\dom(R_{n}) = \ran(T_n)$ as given in \eqref{7.22} is dense in $L^{2}((0,\infty)).$ Indeed, 
one verifies that for any $\alpha>-1,$
\begin{equation}
\big\{x^{\alpha/2}e^{-x/2}L_{m}^{\alpha}(x) \,\big| \, m\in\mathbb{N}_{0}\big\}
\subset \dom(R_{n}), 
\end{equation}
where $\{L_{m}^{\alpha}\}_{m=0}^{\infty}$ is the sequence of Laguerre
polynomials which forms a complete orthogonal set in the Hilbert space
$L^{2}((0,\infty)).$ Indeed, it is clear that for
$j=0,1,\ldots,n-1,$
\begin{equation}
\lim_{x\rightarrow0}(x^{n+\alpha/2}e^{-x/2}L_{m}^{\alpha}(x))^{(j)}=0 \, \text{ and } \, 
(x^{n+\alpha/2}e^{-x/2}L_{m}^{\alpha}(x))^{(n)}\in L^{2}((0,\infty)).
\end{equation}

In the following we will show that the $n$ boundary conditions in \eqref{7.14} can actually be replaced  by the $(n-1)$ $L^2$-conditions
\begin{equation}
x^j g^{(j)} \in L^2((0,\infty)), \quad j=1,\dots,n-1. 
\end{equation}
To prove this we start with the following elementary observations: For $n \in \bbN$ and 
$x \in (0, \infty)$, 
\begin{align}
& (i) \, f \in AC_{loc}((0,\infty)) \, \text{ if and only if } 
\, x^{n}f \in AC_{loc}((0,\infty)).    \lb{7.26} \\[1mm] 
& (ii) \, (x^{n}f(x))^{(k)} = a_{k}(n,k)x^{n}f^{(k)}(x) + \cdots + a_{0}(n,k)x^{n-k}f(x),  \quad 0 \leq k \leq n, 
\no   \\
& \qquad \text{where } \, a_{j}(n,k)= {k \choose j} \prod_{\ell = 0}^{k-j-1}(n - \ell),  \quad 0 \leq j \leq k, 
\lb{7.27} \\[1mm] 
& (iii) \, x^{k}(xf(x))^{(k+1)} = x^{k+1}f^{(k+1)}(x) + (k+1)x^{k}f^{(k)}(x), \quad k \in \bbN_0, 
\lb{7.29} \\[1mm] 
& (iv) \, x(x^{k}f^{(k)}(x))' = x^{k+1}f^{(k+1)}(x) + kx^{k}f^{(k)}(x),  \quad k \in \bbN_0.   \lb{7.30}
\end{align}

\begin{lemma} \lb{l7.3}
Let $n \in \bbN$, then
\begin{align}
& \dom\big(T_{1}^{-n}\big)      \lb{7.31}  \\
& \quad =  \big\{g \in L^{2}((0,\infty)) \, \big| \, g \in AC_{loc}^{(n-1)}((0,\infty)); \, 
 x^{j} g^{(j)} \in L^{2}((0,\infty)), \, j=1, \dots, n \big\}.    \no
\end{align}
\end{lemma}  
\begin{proof}
We start with the case $n=1$ and note that 
\begin{equation} 
g \in L^2((0,\infty)), \, g \in AC_{loc}((0,\infty)), \, \text{ and } \, x g' \in L^2((0,\infty)) 
\, \text{ implies } \,  \lim_{x \downarrow 0} x g(x) =0.   \lb{7.32}
\end{equation}
Indeed, observing
\begin{equation}
\int_{x_0}^x t g'(t) \, dt = [t g(t)]\big|_{x_0}^x - \int_{x_0}^x g(t) \, dt
\end{equation}
shows that $c = \lim_{x \downarrow 0} x g(x)$ exists. If $c \ne 0$, then without loss of generality we can 
assume that $c > 0$. Then $x g(x) > c/2$, equivalently, $g(x) > c/(2x)$ for sufficiently small $0 < x$, yielding $g \notin L^2((0,\infty))$, a contradiction. Thus, $c=0$ and \eqref{7.32} holds. 

Since $(x g)' = x g' + g$, this implies 
\begin{align}
\dom\big(T_{1}^{-1}\big) 
&=  \big\{g \in L^{2}((0,\infty)) \, \big| \, g \in AC_{loc}((0,\infty)); \, x g' \in L^{2}((0,\infty))\big\} 
\lb{7.34}
\end{align}
and hence verifies \eqref{7.31} for $n=1$.

Next, we use induction on $n \in \bbN$. Assume \eqref{7.31} holds for $n \in \bbN$ fixed. 
Then for $n+1$, one obtains 
\begin{align}
& \dom\big(T_{1}^{-(n+1)}\big) = \dom\big(T_{1}^{-n}T_{1}^{-1}\big)    \no \\
& \quad = \big\{g \in \dom\big(T_{1}^{-1}\big) \, \big| \, \big(T_1^{-1} g\big) \in \dom\big(T_{1}^{-n}\big) \big\}    \no \\
& \quad = \big\{g \in \dom\big(T_{1}^{-1}\big) \, \big| \, (x g)' \in \dom\big(T_{1}^{-n}\big) \big\}    \no \\
& \quad = \big\{g \in L^{2}((0,\infty)) \, \big| \, g \in AC_{loc}((0,\infty)); \, x g' \in L^{2}((0,\infty)); 
\, (x g)' \in L^{2}((0,\infty));  \no \\
& \qquad \qquad  (x g)' \in AC_{loc}^{(n-1)}((0,\infty)); \, 
x^{j}(x g)^{(j+1)} \in L^{2}((0,\infty)), \, j=1, \dots, n \big\}    \lb{7.34a} \\
\begin{split}
& \quad = \big\{g \in L^{2}((0,\infty)) \, \big| \, g \in AC_{loc}((0,\infty)); \, 
(x g)' \in AC_{loc}^{(n-1)}((0,\infty));   \lb{7.34b} \\
&\hspace*{4.4cm} x^{j}(x g)^{(j+1)} \in L^{2}((0,\infty)), \, j=0, \dots, n \big\}     
\end{split}  \\
& \quad = \big\{g \in L^{2}((0,\infty)) \, \big| \, g \in AC_{loc}^{(n)}((0,\infty)); \, 
x^{j} g^{(j)} \in L^{2}((0,\infty)), \, j=1, \dots, n+1 \big\},     \lb{7.35} 
\end{align}
as desired. Here we used the induction hypothesis in \eqref{7.34a}, and again the fact 
$(x g)' = x g' + g$ and 
$f \in L^2((0,\infty))$ to conclude that $(x g)' \in L^2((0,\infty))$ if and only if $x g' \in L^2((0,\infty))$ in 
\eqref{7.34b}. To arrive at \eqref{7.35} one uses \eqref{7.26} and hence, 
\begin{equation} \lb{7.36}
g \in AC_{loc}((0,\infty)) \text{ and } (x g)' \in AC_{loc}^{(n-1)}((0,\infty))
\, \text{ if and only if } \, g \in AC_{loc}^{(n)}((0,\infty)), 
\end{equation}
as well as \eqref{7.29} and 
\begin{equation} \lb{7.37}
x^{j}(x g)^{(j+1)} = x^{j+1} g^{(j+1)} + (j+1)x^{j} g^{(j)} \in L^{2}((0,\infty)), \quad j=0, \dots, n,
\end{equation}
which iteratively yields $x^{k} g^{(k)} \in L^{2}((0,\infty))$ for $1 \leq k \leq n+1$.
\end{proof}

\begin{lemma} \lb{l7.4}
Let $n \in \bbN$.
Assume $f \in AC^{(n-1)}_{loc}((0,\infty))$ and $x^{k}f^{(k)} \in L^{2}((0,\infty))$ for $k=0,1, \dots, n$.
Then
\begin{equation} \lb{7.38}
\lim_{x \downarrow 0} (x^{n}f(x))^{(j)} = 0, \quad j=0, 1, \dots, n-1. 
\end{equation} 
\end{lemma}
\begin{proof}
The case $n=1$ holds by \eqref{7.32}.

For $n=2$, assume $f \in AC^{(1)}_{loc}((0,\infty))$ with $f, xf', x^{2}f'' \in L^{2}((0,\infty))$.
Then for $j=0$, one has again by \eqref{7.32},
\begin{equation} \lb{7.39}
\lim_{x \downarrow 0} xf(x) = 0 \, \text{ and hence } \, \lim_{x \downarrow 0} x^2 f(x) = 0.
\end{equation}
For $j=1$, one notes that $xf' \in AC_{loc}((0,\infty))$, see \eqref{7.26}, and 
$x(xf')' = x^{2}f'' + xf' \in L^{2}((0,\infty))$.
Hence, applying \eqref{7.32} to $g=xf'$ yields 
\begin{equation} \lb{7.40}
\lim_{x \downarrow 0} x(xf'(x)) = \lim_{x \downarrow 0} x^{2}f'(x) = 0.
\end{equation}
Combining \eqref{7.39} and \eqref{7.40} shows
\begin{equation}
\lim_{x \downarrow 0} (x^{2}f(x))' = \lim_{x \downarrow 0} \big[x^{2}f'(x) + 2xf(x)\big] = 0,
\end{equation}
proving \eqref{7.38} for $n=2$.

Next, we prove \eqref{7.38} for general $n \in \bbN$. Let $n \in \bbN$ be fixed and assume the hypotheses of the lemma, that is, 
\begin{equation} 
f \in AC^{(n-1)}_{loc}((0,\infty)) \, \text{ and } \, 
x^{k}f^{(k)} \in L^{2}((0,\infty)), \quad k=0,1, \dots, n.    \lb{7.42} 
\end{equation} 
One notes that $f \in AC^{(n-1)}_{loc}((0, \infty))$ implies $x^{j}f^{(j)} \in AC_{loc}((0,\infty))$ for $j=0, \dots, n-1$, see \eqref{7.26}.
Using \eqref{7.30} and \eqref{7.42}, 
\begin{equation}
x(x^{j}f^{(j)})' = x^{j+1}f^{(j+1)} + jx^{j}f^{(j)} \in L^{2}((0,\infty)),  \quad j=0, \dots, n-1.
\end{equation} 
Thus, applying \eqref{7.32} iteratively to $g_{j} := x^{j}f^{(j)}$ one arrives at  
\begin{equation}
\lim_{x \downarrow 0} xg_{j}(x) = \lim_{x \downarrow 0} x^{j+1}f^{(j)}(x) = 0, 
\quad j=0, 1, \dots, n-1.
\end{equation}
In particular, for any $m \in \bbN$ with $m>j$,
\begin{equation} \lb{7.45}
\lim_{x \downarrow 0} x^{m}f^{(j)}(x) = 0, \quad j=0, 1, \dots, n-1.
\end{equation}
Thus, by \eqref{7.27} and \eqref{7.45}, one obtains 
\begin{align}
\lim_{x \downarrow 0} (x^{n}f(x))^{(j)} &= a_{j} \lim_{x \downarrow 0} x^{n}f^{(j)}(x) + \dots + a_{0}\lim_{x \downarrow 0} x^{n-j}f(x)  \no \\
&= a_{j}\cdot 0 + \dots + a_{0} \cdot 0 =0,  \quad j = 0, \dots, n-1.
\end{align} 
\end{proof}

Introducing 
\begin{equation}\label{e1.2}
p_{n}(z) = \prod_{k=0}^{n-1}(z+k), \quad z \in \bbC, \; n \in \bbN, 
\end{equation}
we are now ready to characterize $T_n^{-1}$ in terms of $T_1^{-1}$:

\begin{theorem} \lb{t7.5}
Let $n \in \bbN$, then 
\begin{align} 
& T_{n}^{-1} = p_{n}\big(T_{1}^{-1}\big),   \lb{7.48} \\
& \dom\big(T_{n}^{-1}\big) = \dom\big(T_{1}^{-n}\big)   
=  \big\{g \in L^{2}((0,\infty)) \, \big| \, g \in AC_{loc}^{(n-1)}((0,\infty));   \lb{7.49} \\
& \hspace*{5.3cm} x^{j} g^{(j)} \in L^{2}((0,\infty)), \, j=1, \dots, n \big\}.    \no 
\end{align}
\end{theorem} 
\begin{proof}
Focusing at first on \eqref{7.49} suppose $f \in \dom\big(T_{n}^{-1}\big) = \ran(T_{n})$. 
Then, for some $h \in L^{2}((0,\infty))$,  
\begin{equation} \lb{7.50}
f(x) =  \frac{1}{x^n} \int_{0}^{x} \int_{0}^{t_1} \cdots \int_{0}^{t_{n-1}} h(u) du dt_{n-1} \dots dt_1,
\end{equation}
equivalently,
\begin{equation} \lb{7.51}
x^n f(x) =  \int_{0}^{x} \int_{0}^{t_1} \cdots \int_{0}^{t_{n-1}} h(u) du dt_{n-1} \dots dt_1.
\end{equation}
Taking the $k$-th derivative, $0 \leq k < n$, of \eqref{7.51} yields 
\begin{align} \lb{7.52}
& a_{k}x^{n}f^{(k)}(x) + \dots + a_{0}x^{n-k}f(x)
= \int_{0}^{x} \int_{0}^{t_1} \cdots \int_{0}^{t_{n-k-1}} h(u) du dt_{n-k-1} \dots dt_1,    \no \\
& \hspace*{8.5cm} k = 0, \dots, n-1,
\end{align}
where the coefficients $a_{j}(n,k)$ are given by \eqref{7.27}. 
Dividing \eqref{7.52} by $x^{n-k}$ yields
\begin{align} \lb{7.53}
a_{k}x^{k}f^{(k)} + \dots + a_{1}xf' + a_{0}f
&= T_{n-k} \, h \in L^{2}((0,\infty))
\end{align}
which iteratively proves that $x^{k}f^{(k)} \in L^{2}((0,\infty))$ for $k=0,\dots, n-1$.
Finally, taking the $n$-th derivative of \eqref{7.51} shows
\begin{align} \lb{7.54}
a_{n}x^{n}f^{(n)} + \dots + a_{1}xf' + a_{0}f = h \in L^{2}((0,\infty))
\end{align}
proving $x^{n}f^{(n)} \in L^{2}((0,\infty))$ and thus $f \in \dom\big(T_{1}^{-n}\big)$.

Conversely, suppose $f \in \dom\big(T_{1}^{-n}\big)$. Then $(x^{n}f)^{(n)} \in L^{2}((0,\infty))$ 
since $x^{j}f^{(j)} \in L^{2}((0,\infty))$ for $j=0,\dots, n$ by 
hypothesis, and
\begin{equation}
 (x^{n}f)^{(n)} =a_{n}x^{n}f^{(n)} + \dots + a_{1}xf' + a_{0}f.
\end{equation}
The condition $\lim_{x \downarrow 0}(x^{n}f(x))^{(j)}= 0$, $j=0,\dots, n-1$, follows from 
Lemma \ref{l7.4}.

Turning to the proof of \eqref{7.48}, one notes that the case $n=1$ 
in \eqref{7.48} obviously holds. Hence, assume \eqref{7.48} holds for some $n \in \bbN$ fixed. Then for $n+1$ one computes,
\begin{align}
p_{n+1}(T_1^{-1})f &= \prod_{k=0}^{n}(T_{1}^{-1}+k)f  
= T_{n}^{-1}(T_{1}^{-1} + n)f \text{ by hypothesis } \no \\
&=  T_{n}^{-1}T_{1}^{-1}f+ nT_{n}^{-1}f  \no \\
&= (x^{n}(xf'+f))^{(n)} + n(x^{n}f)^{(n)}  \no \\
&= (x^{n+1}f')^{(n)}+(x^{n}f)^{(n)}+ n(x^{n}f)^{(n)}  \no \\
&= (x^{n+1}f')^{(n)} + ((n+1)x^{n}f)^{(n)}  \no \\
&= ( x^{n+1}f' + (n+1)x^{n}f )^{(n)} \no \\
&= ((x^{n+1}f)')^{(n)} \no \\
&= (x^{n+1}f)^{(n+1)} \no \\
&= T_{n+1}^{-1}f, \quad f \in \dom\big(T_{n+1}^{-1}\big) = \dom\big(T_{1}^{-n-1}\big). 
\end{align}
\end{proof}

Given Theorem \ref{t7.5}, spectral analysis of $T_n$ reduces to that of $T_1$, respectively, 
$T_1^{-1}$, via the spectral mapping theorem. Thus, by \eqref{7.34}, we recall that 
\begin{align}
\begin{split} 
& (T_1^{-1} f)(x) = (x f)'(x) = x f'(x) + f(x), \\
& \, f \in \dom\big(T_1^{-1}\big) = \big\{g \in L^2((0,\infty)) \, \big| \, g \in AC_{loc}((0,\infty); 
(x g)' \in L^2((0,\infty))\big\}.
\end{split} 
\end{align}

Next, we introduce the unitary Mellin transform $\cM$ and its inverse, $\cM^{-1}$, via the pair of formulas
\begin{align}
& \cM \colon \begin{cases} L^2((0,\infty); dx) \to L^2(\bbR; d\lambda), \\[1mm] 
f \mapsto (\cM f)(\lambda) \equiv f^*(\lambda) := (2 \pi)^{-1/2} 
\slim_{a \uparrow \infty} \int_{1/a}^a f(x) x^{- (1/2) + i \lambda} dx \\ 
\hspace*{8.2cm} \text{for a.e.~$\lambda \in \bbR$,}  
\end{cases}   \\
& \cM^{-1} \colon \begin{cases} L^2(\bbR; d\lambda) \to L^2((0,\infty); dx), \\[1mm] 
f^* \mapsto (\cM^{-1} f^*)(x) \equiv f(x) := (2 \pi)^{-1/2} \slim_{b \uparrow \infty} \int_{- b}^b f^*(\lambda) 
x^{- (1/2) - i \lambda} d\lambda \\ 
\hspace*{7.15cm} \text{for a.e.~$x \in (0,\infty)$.} 
\end{cases} 
\end{align}
For details on $\cM$ (resp., $\cM^{-1}$) we refer, for instance, to \cite[Sect.~3.17]{Ti86} 
(see also, \cite[Sect.~1.3]{Ya96}). 

The fact,
\begin{equation}
i \bigg(\f{d}{dx} x - \f{1}{2}\bigg) x^{- (1/2) - i \lambda} = \lambda x^{- (1/2) - i \lambda}, 
\quad x \in (0, \infty), \; \lambda \in \bbR,
\end{equation}
naturally leads to the following definition of the operator $S_1$ in $L^2((0,\infty); dx)$,
\begin{equation}
S_1 := i \big(T_1^{-1} - 2^{-1} I_{L^2((0,\infty))}\big), \quad \dom(S_1) = \dom\big(T_1^{-1}\big),     \lb{B.6} 
\end{equation}
and establishes that $S_1$ is unitarily equivalent to the operator of multiplication by the 
independent variable in $L^2(\bbR)$,
\begin{align}
\begin{split} 
& \big(\cM S_1 \cM^{-1} f^*\big)(\lambda) = \lambda f^*(\lambda) \, \text{ for a.e.~$\lambda \in \bbR$ and}     \\
& \quad \text{for all $f^* \in L^2(\bbR; d\lambda)$ such that $\lambda f^* \in L^2(\bbR; d\lambda)$.} 
\end{split} 
\end{align}

Summarizing, the Mellin transform diagonalizes $S_1$ and hence $T_1$. Denoting by 
$C(z_0; r_0) \subset \bbC$ the circle of radius $r_0 > 0$ centered at $z_0 \in \bbC$, one obtains 
the following result. 

\begin{theorem} \lb{t7.6} 
Introduce $S_1$ in $L^2((0,\infty);dx)$ as in \eqref{B.6}. Then $S_1$ is self-adjoint and hence 
$T_1$ is normal. Moreover, the spectra of $S_1$ and $T_1$ are simple and purely absolutely continuous. In particular,
\begin{align}
& \sigma(S_1) = \sigma_{ac}(S_1) = \bbR, \quad \sigma_p(S_1) = \sigma_{sc}(S_1) = \emptyset, \\ 
&\sigma(T_1) = \sigma_{ac}(T_1) = C(1;1), \quad \sigma_p(T_1) = \sigma_{sc}(T_1) = \emptyset. 
\end{align}
\end{theorem}
\begin{proof}
The claims concerning $S_1$ follow from \eqref{B.6}, unitarity of $\cM$, and the fact that the operator 
of multiplication by the independent variable in $L^2(\bbR; d\lambda)$ has purely absolutely continuous spectrum. The pair of maps 
\begin{equation}
C(1;1) \ni z \mapsto i \big(z^{-1} - 2^{-1}\big) \in \bbR, \quad \bbR \ni \lambda \mapsto 
\big[- i \lambda + 2^{-1}\big]^{-1} \in C(1;1),  
\end{equation}
establishes the facts concerning $T_1$. 
\end{proof}

\begin{remark} \lb{r7.7}
Introducing the family of operators 
\begin{equation}
(T_{1,z} f)(x) := \int_0^1 s^{-z} f(st) \, dt = x^{z-1} \int_0^x u^{-z} f(u) \, du, 
\quad \Re(z) < 1/2, 
\end{equation}
one verifies that ($\Re(z) < 1/2$) 
\begin{align}
& T_{1,0} = T_1, \\ 
& \bigg(x^z \f{d}{dx} x^{1-z} T_{1,z} f\bigg)(x) = f(x), \\
& x^z \f{d}{dx} x^{1-z} = x \f{d}{dx} + I - z I= \f{d}{dx} x - z I, \\
& T_{1,z} = \bigg(\f{d}{dx} x - z I\bigg)^{-1} = \big(T_1^{-1} - z I\big)^{-1} = (I - z T_1)^{-1} T_1,  \no \\
& \hspace*{6mm} = - z^{-1} I + z^{-2} \big(z^{-1} I - T_1\big)^{-1}, \\
& (T_1 - z I)^{-1} = - z^{-1} I - z^{-2} T_{1, z^{-1}}. 
\end{align}
\hfill $\diamond$ 
\end{remark}

The fact, $ \sigma(T_1) = C(1;1)$, as well as the resolvent formula \eqref{7.71} for $T_1$ are well-known, we refer, for instance, to \cite{BHS65}, and \cite{Bo68} (see also \cite{ABR15}, \cite{GL-S09}, 
\cite{LL-SPZ15}, \cite{Le73}, and the references cited therein). What appears to be less well-known is the a.c. nature of the spectrum of $T_1$ and the spectral representation in terms of the Mellin 
transform. Much of the work on the spectral theory for $T_1$ focused on $p$-dependence of the 
spectrum in $L^p$-spaces (on finite intervals and on the half-line), Hardy, Bergman, and Dirichlet  spaces, etc. For related classes of integral operators see, for instance, \cite{Me98}, \cite{NS94},  
\cite{Pr00}, and the references cited therein.

\begin{remark} \lb{r7.8}
One notes the curious fact that while the closed, symmetric operator $A_1$ in $L^2((0,\infty))$, defined by 
\begin{align}
& (A_1 f)(x) = i f'(x),  \quad  f \in \dom(A_1) = \big\{g \in L^2((0,\infty)) \, \big | \, 
g \in AC_{loc}([0,\infty));     \no \\ 
& \hspace*{6cm}  g(0_+) = 0; \, g' \in L^2((0,\infty))\big\}, 
\end{align}
is the prime example of a symmetric operator with unequal deficiency indices ($1$ and $0$), and hence has no self-adjoint extensions in $L^2((0,\infty))$, the right multiplication of $i d/dx$ with $x$ in \eqref{B.6}, followed by the shift 
$- i/2$, yields the self-adjoint operator $S_1$ in $L^2((0,\infty))$. 
\hfill $\diamond$ 
\end{remark}

Formula \eqref{7.48}, $T_n^{-1} = p_n\big(T_1^{-1}\big)$, $n \in \bbN$, together with the spectral theorem applied to $T_1$, then yield 
\begin{equation}
\sigma\big(T_n^{-1}\big) = p_n\big(\sigma\big(T_1^{-1}\big)\big), \quad n \in \bbN.
\end{equation}

Equivalently, introducing the rational function $r_n$ by 
\begin{equation}
r_n (z) = z^n \prod_{k = 1}^{n-1} (1 + kz)^{-1}, \quad 
z \in \bbC \backslash \big\{- \ell^{-1}\big\}_{1 \leq \ell \leq n-1}, \; n \in \bbN, 
\end{equation} 
the formula
\begin{equation}
T_n = r_n (T_1), \quad n \in \bbN,    \lb{7.69} 
\end{equation}
yields the following facts.

\begin{theorem} \lb{t7.9} 
Let $n \in \bbN$. Then, $T_n$ is normal and 
\begin{align}
& \sigma(T_n) = r_n(\sigma(T_1)) = \big\{r_n\big(1+e^{i \theta}\big) \, \big| \, \theta \in [0, 2 \pi]\big\}, 
\quad n \in \bbN,    \lb{7.70} \\
& \sigma(T_n) = \sigma_{ac}(T_n), \quad \sigma_p(T_n) = \sigma_{sc}(T_n) = \emptyset. 
\lb{7.71}
\end{align}
\end{theorem}
\begin{proof}
Normality of $T_n$ is clear from \eqref{7.69}. 
The facts \eqref{7.70}, \eqref{7.71} follow from combining Theorem \ref{t7.6}, \eqref{7.69}, and the 
spectral theorem for normal operators. 
\end{proof}

We have not been able to find discussions of $T_n$, $n \geq 2$, in the literature.

The spectrum of $T_n$, for various values of $n \in \bbN$, is illustrated next:

\setcounter{figure}{0}

\begin{figure}[H]
\includegraphics[scale=.92, center]{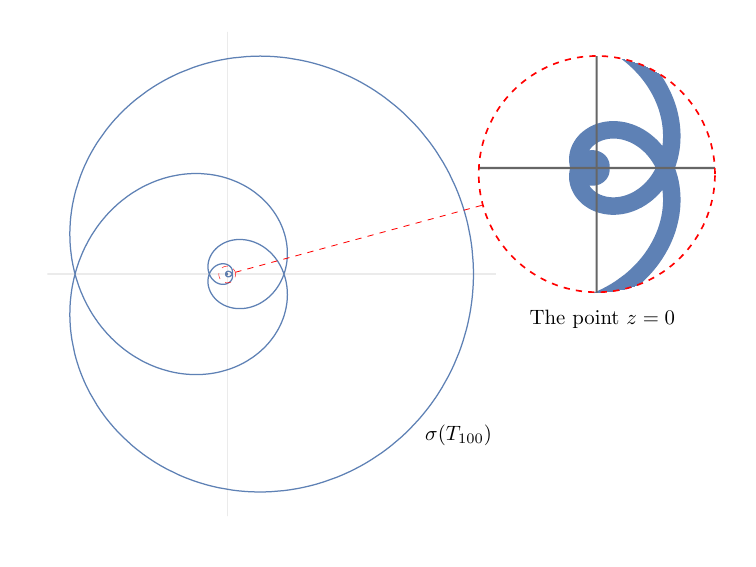}
\caption{$14\times$ Magnification of $\sigma(T_{100})$}
\end{figure}

\begin{figure}[H]
\includegraphics[scale=1, center]{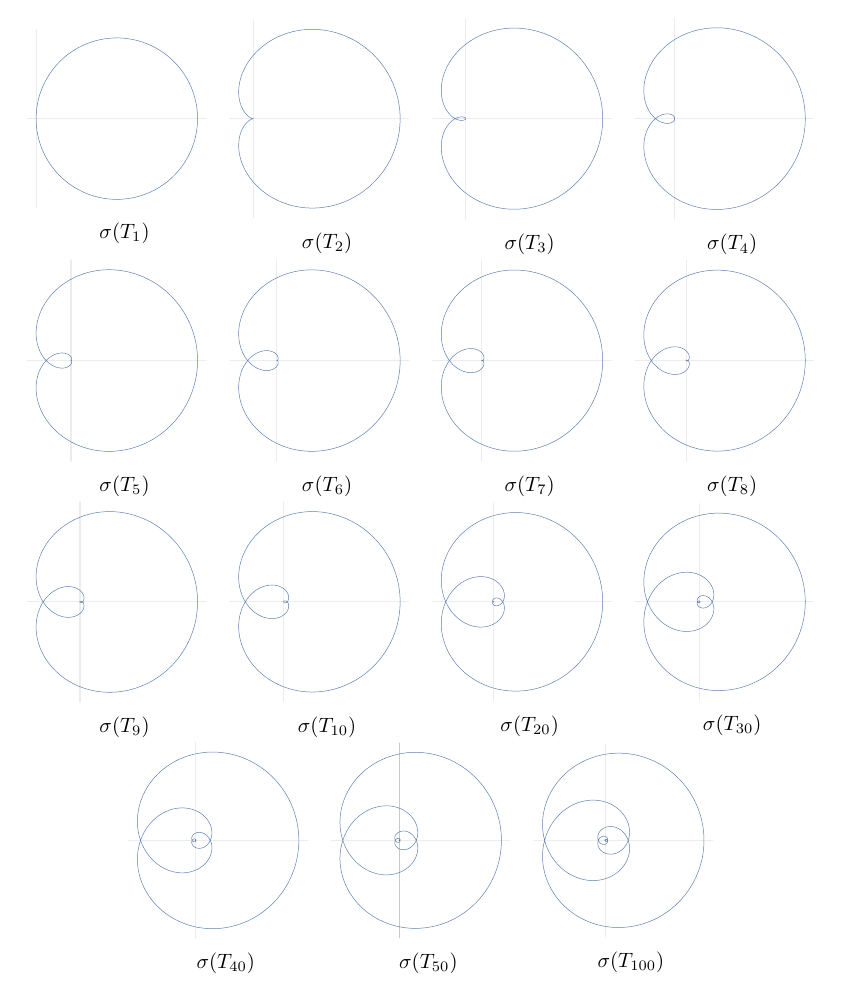}
\caption{The Spectrum of $T_n$ for certain $n \in \bbN$}
\end{figure}

\section{The Birman Inequalities on the Finite Interval $[0,b]$} 
\lb{s8}

For fixed $b\in (0,\infty)$ and $n$ $\in\bbN,$ introduce the set\footnote{It is possible to replace the 
boundary conditions at $x=0$ by $f/x^n \in L^2$ and/or the one at $x=b$ by $f/(b-x)^n \in L^2$, leading to additional spaces $H_n(0,b])'$, $H_n([0,b))'$, and $H_n((0,b))'$ in analogy to \eqref{3.0a}, but we omit further details at this point.} 
\begin{align}
\begin{split} 
H_{n}([0,b]):= \big\{f:[0,b]\rightarrow\bbC \,\big| \,  f^{(n)}\in L^{2}((0,b)); \, f^{(j)}\in AC([0,b]); &\\
f ^{(j)}(0)=f^{(j)}(b)=0, \, j=0,1,\dots,n-1&\big\},
\end{split} 
\end{align}
with associated inner product $(\, \cdot \, ,\, \cdot \,)_{H_{n}([0,b])}$,
\begin{equation}
(f,g)_{H_{n}([0,b])}:=\int_{0}^{b} \ol{f^{(n)}(x)} \, g^{(n)}(x) \, dx, \quad f,g\in
H_{n}([0,b]),     \lb{8.2} 
\end{equation}
and norm
\begin{equation}
\left\Vert f\right\Vert _{H_{n}([0,b])}=\big\Vert f^{(n)}\big\Vert _{L^2((0,b))},   \quad 
f\in H_{n}([0,b]).     \lb{8.3} 
\end{equation}

Next, we state the following result which yields the Birman inequalities on
$[0,b]$ as well as sharpness and equality results on 
$H_{n}([0,b])= H^n_0((0,b))$,
where $H^n_0((0,b))$ denotes the standard Sobolev space on $(0,b)$ obtained 
upon completion of $C_0^{\infty}((0,b))$ in the norm of $H^n((0,b))$, that is, 
\begin{align}
& H^n((0,b)) = \big\{f: [0,b] \to \bbC \, \big| \, f^{(j)} \in AC([0,b]), \, j =0,1, \dots, n-1;  \no \\
& \hspace*{5.15cm}  f^{(k)} \in L^2((0,b)), \, k = 0,1,\dots,n \big\},   \\
& H^n_0((0,b)) = \big\{f \in H^n((0,b)) \, \big| \,  f^{(j)}(0) =  f^{(j)}(b) = 0, \, j =0,1, \dots, n-1\big\}. 
\end{align}

\begin{theorem} \lb{t8.2}
Let $n\in\bbN$, then the following items $(i)$--$(iv)$ hold: \\[1mm] 
$(i)$ For each $n \in \bbN$, and $b \in (0, \infty)$, 
\begin{equation} 
H_{n}([0,b]) = H^n_0((0,b))   \lb{8.7}
\end{equation} 
as sets. In particular, 
\begin{equation}
f\in H_{n}([0,b]) \, \text{ implies } \, f^{(j)}\in L^{2}((0,b)), \quad j=0,1,\dots,n.   \lb{8.8} 
\end{equation} 
In addition, the norms in $H_{n}([0,b])$ and $H^n_0((0,b))$ are equivalent. 
\\[1mm] 
$(ii)$ The following hold: \\[1mm] 
$(\alpha)$ Let $a, c \in [0,\infty)$, $a < c$, $f \colon [a,c] \to \bbC$, with $f^{(j)} \in AC([a,c])$, $f ^{(j)}(a)=0$, $j=0,1,\dots,n-1$, and 
$f^{(n)}\in L^{2}((a,c))$. Then,
\begin{equation}
\int_{a}^{c}\big|f^{(n)}(x)\big|^{2} \, dx\geq\dfrac{[(2n-1)!!]^{2}
}{2^{2n}}\int_{a}^{c} \dfrac{|f(x)|^2}{(x-a)^{2n}} \, dx.   \lb{8.8a}
\end{equation}
$(\beta)$ Let $a, c \in [0,\infty)$, $a < c$, $f:[a,c] \rightarrow \bbC$, with $f^{(j)} \in AC([a,c])$, $f ^{(j)}(c)=0$, $j=0,1,\dots,n-1$, 
and $f^{(n)}\in L^{2}((a,c))$. Then,
\begin{equation}
\int_{a}^{c}\big|f^{(n)}(x)\big|^{2} \, dx\geq\dfrac{[(2n-1)!!]^{2}
}{2^{2n}}\int_{a}^{c} \dfrac{|f(x)|^2}{(c-x)^{2n}} \, dx.  \lb{8.8b}
\end{equation}
$(\gamma)$ Introducing the distance of $x \in (0,b)$ to the boundary $\{0,b\}$ of $(0,b)$ by 
\begin{equation}
d(x) = \min \{x, |b-x|\}, \quad x \in (0,b), \; b \in (0, \infty), 
\end{equation}
one has 
\begin{equation}
\int_{0}^{b}\big|f^{(n)}(x)\big|^{2} \, dx\geq\dfrac{[(2n-1)!!]^{2}
}{2^{2n}}\int_{0}^{b} \dfrac{|f(x)|^2}{d(x)^{2n}} \, dx, \quad f\in H^{n}_0((0,b)).  \lb{8.9}
\end{equation}
In all cases $(\alpha)$--$(\gamma)$, if $f \not \equiv 0$, the inequalities \eqref{8.8a}, \eqref{8.8b}, and 
\eqref{8.9} are strict. \\[1mm] 
$(iii)$ The constant $[(2n-1)!!]^{2}/2^{2n}$ is sharp in all cases $(\alpha)$--$(\gamma)$ in item $(ii)$. 
\end{theorem}
\begin{proof}
$(i)$ Equality of $H_n([0,b])$ with the standard Sobolev space $H^n_0((0,b))$ in \eqref{8.7} (and hence the fact \eqref{8.8}) follows from \cite[p.~29]{Bu98} (discussing the endpoint behavior of 
$f \in H^n((0,b))$ at 
$\{0, b\}$) and especially, from \cite[Theorem~2, p.~127 and Corollary~6, p.~128]{Bu98}. 
Alternatively, one can exploit the boundary conditions $f^{(j)} (0) = f^{(j)} (b) = 0$, 
$0 \leq j \leq n$, and combine \cite[Corollary~V.3.21]{EE89} and the Friedrichs inequality 
\cite[p.~242]{EE89},
\begin{equation}
\big\|f^{(j)}\big\|_{L^2((0,b))} \leq C \big\|f^{(n)}\big\|_{L^2((0,b))}, \quad f \in H^n_0((0,b)), 
\end{equation}
with $C = C(j,n,b) \in (0,\infty)$ independent of $f \in H^n_0((0,b))$. 

For the rest of the proof we assume that $f$ is real-valued. 

To prove item $(ii)$ part $(\alpha)$ one first follows the proof of Theorem \ref{t4.4}, observing that for $k \in \bbN$,
\begin{align}
\int_{a}^{c} \frac{f(x)^2}{(x-a)^{2k}} \, dx
&  =-\frac{1}{2k-1}\left(  \left.  \frac{f(x)^2}{(x-a)^{2k-1}}\right\vert
_{a}^{c}-2\int_{a}^{c}\frac{f(x)f^{\prime}(x)}{(x-a)^{2k-1}} \, dx\right)
\text{ } \no \\
& =-\frac{1}{2k-1}\left(   \frac{f(c)^2}{(c-a)^{2k-1}} -2\int_{a}^{c}\frac{f(x)f^{\prime}(x)}{(x-a)^{2k-1}} \, dx\right)
\text{ } \no \\
&  \leq \frac{2}{2k-1}\int_{a}^{c}\frac{f(x)f^{\prime}(x)}{(x-a)^{2k-1}}dx, 
\end{align}
and then continues as in \eqref{4.14}.  \\
Part $(\beta)$ follows from $(\alpha)$ by reflecting about the interval midpoint.  \\
For part $(\gamma),$ one can follow the argument provided in \cite[Corollary~5.3.2]{Da95} in the context of the Hardy inequality $n=1$: Splitting the interval $(0,b)$ into $(0,b/2] \cup [b/2,b)$, exploiting the fact 
\begin{equation}
\int_{0}^{b/2} \dfrac{[f(x)]^2}{x^{2n}} + \int_{b/2}^b \dfrac{[f(x)]^2}{(b-x)^{2n}} 
= \int_{0}^{b/2} \dfrac{[f(x)]^2}{d(x)^{2n}} + \int_{b/2}^b \dfrac{[f(x)]^2}{d(x)^{2n}} 
= \int_{0}^{b} \dfrac{[f(x)]^2}{d(x)^{2n}},     \lb{8.13} 
\end{equation}
and then separately applying parts $(\alpha)$ to $(0,b/2)$, and $(\beta)$ to $(b/2,b)$, yields \eqref{8.9}. 

To prove strict inequality in \eqref{8.8a}--\eqref{8.9}, it suffices to consider part $(\alpha)$ 
since $(\beta)$ follows by reflection and either $(\alpha)$ or $(\beta)$ implies $(\gamma$). 
One infers from \eqref{4.5c} that functions that would yield equality are of the type 
\begin{equation}
g_0(x) = c_{n-1} (x-a)^{\lambda + n-1} + c_{n-2} (x-a)^{n-2} +  c_{n-3}(x-a)^{n-3} + \dots +  c_{1}(x-a) + c_{0}.
\end{equation}
The fact $g_0^{(j)}(a)=0$ for $j=0,1, \dots, n-2$ shows
\begin{equation}
c_0 = c_1 = \dots = c_{n-2} = 0,
\end{equation}
so that
\begin{equation}
g_0(x) = c_{n-1}(x-a)^{\lambda + n-1}. \lb{0.10}
\end{equation}
Equation \eqref{0.10} suggests equality in \eqref{8.8a} holds only for functions of the form
\begin{equation}
g_0(x) = d (x-a)^{\mu} 
\end{equation}
for some $d \in \C$, $\mu \in \bbR$. To prove $d=0$, we will argue as follows. First, one 
notes that $g_0^{(n)} \in L^{2}((a,c))$ implies
\begin{equation}
2(\mu - n) > -1 \, \text{ or, } \,  \mu > n - 1/2.    \lb{0.14}
\end{equation} 
Inductively, one sees that
\begin{equation}
[\mu(\mu - 1) \cdots (\mu - n+1)]^2 > \frac{[(2n-1)!!]^{2}}{2^{2n}}, \quad n \in \bbN.
\end{equation}
Computing the left side of \eqref{8.8a} then yields
\begin{align}
\int_{a}^{c} \big[g_0^{(n)}(x)\big]^{2}dx 
&= |d|^2 [\mu(\mu - 1) \cdots (\mu - n+1)]^2 \int_{a}^{c} (x-a)^{2(\mu -n)} dx  \\
&> |d|^2 \frac{[(2n-1)!!]^{2}}{2^{2n}}\int_{a}^{c} (x-a)^{2(\mu -n)} dx \\
&= |d|^2 \frac{[(2n-1)!!]^{2}}{2^{2n}}\int_{a}^{c} \frac{[g_0(x)]^2}{(x-a)^{2n}} \, dx,
\end{align}
contradicting equality in \eqref{8.8a}, unless $d=0$.

To prove item $(iii)$ for case $(\alpha)$ one chooses $f_{\sigma}(x) = (x-a)^{\sigma} \chi_{(a,c)}$ and then proceeds as in the proof of Theorem \ref{t6.1}. To settle case $(\beta)$ one uses case $(\alpha)$ combined with reflection about the interval midpoint. If in case $(\gamma)$, the constant $[(2n-1)!!]^{2}/2^{2n}$ would not be optimal and a larger constant should exist, then considering the two intervals $(0,b/2)$ and 
$(b/2,b)$ would lead to a larger constant on at least one of them as well, contradicting cases $(\alpha)$ or 
$(\beta)$. 
\end{proof}

\section{The Vector-Valued Case} 
\lb{s9}

In this section we indicate that all results described thus far extend to the vector-valued 
case in which $f$ is not just complex-valued, but actually, $\cH$-valued, with $\cH$ a separable, complex Hilbert space.

To set the stage we briefly review some facts on Bochner integrability and associated vector-valued $L^p$- and Sobolev spaces. 

Regarding details of the Bochner integral we refer, for instance, to \cite[p.\ 6--21]{ABHN01},
\cite[p.\ 44--50]{DU77}, \cite[p.\ 71--86]{HP85}, \cite[Sect.~4.2]{Ku78}, \cite[Ch.\ III]{Mi78}, 
\cite[Sect.\ V.5]{Yo80}. 
In particular, if $(a,b) \subseteq \bbR$ is a finite or infinite interval and $\cB$ a Banach space, 
and if $p\ge 1$, the symbol $L^p((a,b);dx;\cB)$, in short, $L^p((a,b);\cB)$, whenever Lebesgue measure is understood, denotes the set of equivalence classes of strongly measurable $\cB$-valued functions which differ at most on sets of Lebesgue measure zero, such that 
$\|f(\cdot)\|_{\cB}^p \in L^1((a,b))$. The
corresponding norm in $L^p((a,b);\cB)$ is given by
\begin{equation}
\|f\|_{L^p((a,b);\cB)} = \bigg(\int_{(a,b)} \|f(x)\|_{\cB}^p \, dx\bigg)^{1/p}
\end{equation}
and $L^p((a,b);\cB)$ is a Banach space.

If $\cH$ is a separable Hilbert space, then so is $L^2((a,b);\cH)$ (see, e.g.,
\cite[Subsects.\ 4.3.1, 4.3.2]{BW83}, \cite[Sect.\ 7.1]{BS87}).

One recalls that by a result of Pettis \cite{Pe38}, if $\cB$ is separable, weak
measurability of $\cB$-valued functions implies their strong measurability.

A map $f:[c,d] \to \cB$ (with $[c,d] \subset (a,b)$) is called \textit{absolutely continuous 
on $[c,d]$}, denoted by $f \in AC([c,d]; \cB)$, if 
\begin{equation}
f(x)= f(x_0) + \int_{x_0}^x g(t) \, dt, \quad x_0, x \in [c,d], 
\end{equation}
for some $g \in L^1((c,d);\cB)$. In particular, $f$ is then 
strongly differentiable a.e.\ on $(c,d)$ and
\begin{equation}
f'(x) = g(x) \, \text{ for a.e.\ $x \in (c,d)$}.
\end{equation}
Similarly, $f:[c,d] \to \cB$ is called \textit{locally absolutely continuous}, denoted by 
$f \in AC_{loc}([c,d]; \cB)$, if $f \in AC([c',d']; \cB)$ on any closed subinterval $[c',d'] \subset (c,d)$. 

Sobolev spaces $W^{n,p}((a,b); \cB)$ for $n\in\bbN$ and $p\geq 1$ are defined as follows: $W^{1,p}((a,b);\cB)$ is the set of all
$f\in L^p((a,b);\cB)$ such that there exists a $g\in L^p((a,b);\cB)$ and an
$x_0\in(a,b)$ such that
\begin{equation}
f(x)=f(x_0)+\int_{x_0}^x g(t) \, dt  \, \text{ for a.e.\ $x \in (a,b)$.}
\end{equation}
In this case $g$ is the strong derivative of $f$, $g=f'$. Similarly,
$W^{n,p}((a,b);\cB)$ is the set of all $f\in L^p((a,b);\cB)$ so that the first $n$ strong
derivatives of $f$ are in $L^p((a,b);\cB)$. Finally, $W^{n,p}_{\rm loc}((a,b);\cB)$ is
the set of $\cB$-valued functions defined on $(a,b)$ for which the restrictions to any
compact interval $[\alpha,\beta]\subset(a,b)$ are in $W^{n,p}((\alpha,\beta);\cB)$.
In particular, this applies to the case $n=0$ and thus defines $L^p_{\rm loc}((a,b);\cB)$.
If $a$ is finite we may allow $[\alpha,\beta]$ to be a subset of $[a,b)$ and denote the
resulting space by $W^{n,p}_{\rm loc}([a,b);\cB)$ (and again this applies to the case
$n=0$).

Following a frequent practice (cf., e.g., the discussion in \cite[Sect.\ III.1.2]{Am95}), we
will call elements of $W^{1,1} ([c,d];\cB)$, $[c,d] \subset (a,b)$ (resp.,
$W^{1,1}_{\rm loc}((a,b);dx;\cB)$), strongly absolutely continuous $\cB$-valued functions
on $[c,d]$ (resp., strongly locally absolutely continuous $\cB$-valued functions
on $(a,b)$), but caution the reader that unless $\cB$ possesses the Radon--Nikodym
(RN) property, this notion differs from the classical definition
of $\cB$-valued absolutely continuous functions (we refer the interested reader
to \cite[Sect.\ VII.6]{DU77} for an extensive list of conditions equivalent to $\cB$ having the
RN property). Here we just mention that reflexivity of $\cB$ implies the RN property.

In the special case $\cB = \bbC$, we omit $\cB$ and just write
$L^p_{(loc)}((a,b))$, as usual.

In the following we will typically employ the special case $p=2$ and use a complex, separable Hilbert space $\cH$ for $\cB$, denoting the corresponding Sobolev spaces by $H^{n}((a,b);\cH)$. The inner product in $L^2((a,b); \cH)$, in obvious notation, then reads
\begin{equation}
(f, g)_{L^2((a,b); \cH)} = \int_a^b (f(x), g(x))_{\cH} \, dx, \quad f, g \in L^2((a,b); \cH). 
\end{equation}
In other words, $L^2((a,b); \cH)$ can be identified with the constant fiber direct integral of Hilbert spaces,
\begin{equation}
L^2((a,b); \cH) \simeq \int^{\oplus}_{(a,b)} \cH \, dx. 
\end{equation}

For applications of these concepts to Schr\"odinger operators with operator-valued potentials we refer to \cite{GWZ13}; applications to scattering theory for multi-dimensional Schr\"odinger operators are studied in great detail in \cite[Chs.~IV, V]{Ku78}. The latter reference motivated us to add the present section. 

Before stating the sequence of Birman inequalities in the $\cH$-valued context, we recall a few basic properties of Bochner integrals which illustrate why all results in Sections \ref{s2}--\ref{s8} 
in the special complex-valued case (i.e., $\cH = \bbC$) carry over verbatim to the vector-valued situation.

As representative examples we mention, for instance, 
\begin{align}
&\bigg\|\int_{(a,b)} f(x) \, dx \bigg\|_{\cH} \leq \int_{(a,b)} \|f(x)\|_{\cH} \, dx, \quad 
f \in L^1((a,b); \cH),   \\[1mm] 
& \|f g\|_{L^1((a,b); \cH)} \leq \|f\|_{L^2((a,b); \cH)} \|g\|_{L^2((a,b))}, \quad f \in L^2((a,b); \cH), \; 
g \in L^2((a,b)),    \no \\
& \quad \text{with equality if and only if for some $(0,0) \neq (\alpha, \beta) \in \bbR^2$,}    \no \\ 
& \quad \text{$\alpha \|f(x)\|^2_{\cH} = \beta |g(x)|^2$ for a.e.~$x \in (a,b)$,} \\
& \|(f(x), g(x))_{\cH}\|_{L^1((a,b))} \leq \|f\|_{L^2((a,b);\cH)} \|g\|_{L^2((a,b);\cH)}, \quad 
f,g \in L^2((a,b); \cH),    \no \\
& \quad \text{with equality if and only if for some $(0,0) \neq (\alpha, \beta) \in \bbR^2$,} \no \\ 
& \quad \text{$\alpha \|f(x)\|^2_{\cH} = \beta \|g(x)\|^2_{\cH}$ for a.e.~$x \in (a,b)$,} \\
& \int_c^d (f'(x), g(x))_{\cH} \, dx + \int_c^d (f(x), g'(x))_{\cH} \, dx = (f(d), g(d))_{\cH} 
- (f(c), g(c))_{\cH},    \no \\
& \hspace*{6cm} f,g \in AC([a,b]; \cH), \; (c,d) \subseteq (a,b).
\end{align}

Given these preliminaries we now introduce the spaces 
\begin{align}
& H_{n}([0,\infty); \cH):= \big\{f:[0,\infty)\rightarrow \cH  \, \big| \,  f^{(j)}\in 
AC_{loc}([0,\infty); \cH); \, f^{(n)}\in L^{2}((0,\infty); \cH);    \no \\
& \hspace*{4.7cm}  f^{(j)}(0)=0, \, j=0,1,\dots,n-1\big\}, \quad n \in \bbN.   \lb{9.14} 
\end{align}
As in the scalar context, the space $H_{n}([0,\infty); \cH)$ is a Hilbert space when endowed 
with the inner product 
\begin{equation}
(f,g)_{H_{n}([0,\infty); \cH)}:=\int_{0}^{\infty} \big(f^{(n)}(x), g^{(n)}(x)\big)_{\cH} \, dx, \quad f,g\in
H_{n}([0,\infty); \cH), \lb{9.15} 
\end{equation}
and norm 
\begin{equation}
\|f\| _{H_{n}([0,\infty); \cH)}=\big\| f^{(n)}\big\| _{L^2((0,\infty); \cH)},   \quad 
f\in H_{n}([0,\infty); \cH).     \lb{9.15a} 
\end{equation} 
Similarly, introducing 
\begin{align}
& \hatt H_{n}((0,\infty); \cH):= \big\{f:(0,\infty)\rightarrow \cH \, \big| \, f^{(j)}\in
AC_{loc}((0,\infty); \cH), \, j=0,1,\dots,n-1;  \no \\
& \hspace*{5.2cm} f^{(n)},f/x^{n}\in L^{2}((0,\infty); \cH)\big\}, \quad n \in \bbN,   \lb{9.16} 
\end{align}
one proves as in the scalar context that
\begin{equation}
H_{n}([0,\infty); \cH)= \hatt H_{n}((0,\infty); \cH), \quad n \in \bbN,     \lb{9.17} 
\end{equation}
and that $C_0^{\infty}((0,\infty); \cH)$ is dense in $H_{n}([0,\infty); \cH)$. For the latter assertion it suffices to replace \eqref{3.6} by 
\begin{equation}
T_{\ol{g_0}}(\varphi):= \int_{0}^{\infty} (g_{0}(x), \varphi(x))_{\cH} \, dx,   
\quad \varphi \in C_0^{\infty}((0,\infty); \cH).    \lb{9.18} 
\end{equation}

These facts are shown as in the scalar context upon introducing the a.e.\ nonnegative 
Lebesgue measurable weight function $w: (a,b)\rightarrow\bbR$ and 
$\varphi,\psi: (a,b)\rightarrow\bbC $ Lebesgue measurable functions satisfying conditions 
$(i)$--$(iii)$ in Theorem \ref{t2.1}.

Define the linear operators $A, B:L^{2}((a,b);wdx; \cH)\rightarrow 
L_{loc}^{2}((a,b);wdx; \cH)$ by
\begin{equation}
(Af)(x):=\varphi(x)\int_{x}^{b}\psi(t)f(t)w(t)dt, \quad f\in L_{loc}^{2}((a,b);wdx; \cH), 
\end{equation}
and
\begin{equation}
(Bf)(x):=\psi(x)\int_{a}^{x}\varphi(t)f(t)w(t)dt, \quad f\in L_{loc}^{2}((a,b);wdx; \cH),
\end{equation}
and the function $K:(a,b)\rightarrow \bbR $ by
\begin{equation}
K(x):=\left(  \int_{a}^{x}\left\vert \varphi(t)\right\vert ^{2}w(t)dt\right)
^{1/2}\left(  \int_{x}^{b}\left\vert \psi(t)\right\vert ^{2}w(t)dt\right)
^{1/2}.
\end{equation}
Then again $A$ and $B$ are bounded linear operators in $L^{2}((a,b);wdx; \cH)$ if and only if
\begin{equation}
K:=\sup_{x\in (a,b)}K(x)<\infty.
\end{equation}
Moreover, if $K < \infty$, then $A$ and $B$ are adjoints of each other in 
$L^{2}((a,b);wdx; \cH)$, with 
\begin{align}
\begin{split} 
\|Af\|_{L^{2}((a,b);wdx; \cH)} = \|Bf\| _{L^{2}((a,b);wdx; \cH)} 
\leq 2K\|f\| _{L^{2}((a,b);wdx; \cH)},&    \\
\quad f\in L^{2}((a,b);wdx; \cH).&    \lb{9.23}
\end{split} 
\end{align}
Furthermore, the constant
$2K$ in \eqref{2.7} is best possible; that is,
\begin{equation}
\| A\|_{\cB(L^2((a,b); wdx; \cH))}= \|B\|_{\cB(L^2((a,b); wdx; \cH))} =2K. \lb{9.24}
\end{equation}

Also Theorem \ref{t3.2} extends to the present vector-valued situation in the following form: 
Let $f\in H_{n}([0,\infty); \cH)$, then
\begin{align}
& f^{(n-j)}/x^{j}\in L^{2}((0,\infty); \cH), \quad j=0,1,\dots, n,    \\
& \lim_{x\uparrow\infty}\dfrac{\|f^{(j)}(x)\|^{2}_{\cH}}{x^{2n-2j-1}}=0, 
\quad j=0,1,\dots,n-1,    \\
& \lim_{x\downarrow 0}\dfrac{\|f^{(j)}(x)\|^{2}_{\cH}}{x^{2n-2j-1}}=0, 
\quad j=0,1,\dots,n-1.
\end{align}

With these preparations at hand, we can now formulate the vector-valued extension of 
Birman's sequence of Hardy--Rellich-type inequality: 
\begin{theorem} \lb{t9.1} 
For $0\neq f\in H_{n}([0,\infty); \cH),$ one has 
\begin{equation}
\int_{0}^{\infty} \big\|f^{(n)}(x)\big\|^{2}_{\cH} \, dx > 
\frac{[(2n-1)!!]^{2}}{2^{2n}}\int_{0}^{\infty} \frac{\|f(x)\|^2_{\cH}}{x^{2n}} \, dx, \quad n \in \bbN. 
\lb{9.28}
\end{equation}
\end{theorem}
For the proof of Theorem \ref{t9.1} one can now follow either of the two proofs given in Section \ref{s4}. Optimality of the constants is clear from the special case $\cH = \bbC$. 

The special case $n=1$, that is, the $\cH$-valued Hardy inequality (in fact, a weighted version of 
the latter) appeared in \cite[p.~4.50]{Ku78}. 

It remains to discuss the finite interval case $(0, b)$, $b \in (0,\infty)$. We start by introducing the 
set\footnote{Again, as noted at the beginning of Section \ref{s8}, one could introduce analogous 
spaces with the boundary conditions at $x=0$ 
and/or $x=b$ replaced by $L^2$-conditions for $\|f(\, \cdot \,)\|_{\cH}/x^n$ and 
$\|f(\, \cdot \,)\|_{\cH}/(b-x)^n$, respectively, but refrain from doing so at this point.}  
(with $n \in \bbN$),  
\begin{align} 
\begin{split} 
&H_{n}([0,b]; \cH):= \big\{f:[0,b]\rightarrow \cH \,\big| \,  f^{(n)}\in L^{2}((0,b); \cH); \, 
f^{(j)}\in AC([0,b]; \cH);    \\
& \hspace*{5cm} f ^{(j)}(0)=f^{(j)}(b)=0, \, j=0,1,\dots,n-1\big\},
\end{split} 
\end{align}
and the standard $\cH$-valued Sobolev spaces, 
\begin{align}
\begin{split}
& H^n((0,b); \cH) = \big\{f: [0,b] \to \cH \, \big| \, f^{(j)} \in AC([0,b]; \cH), \, j =0,1, \dots, n-1;     \\
& \hspace*{5.6cm}  f^{(k)} \in L^2((0,b); \cH), \, k = 0,1,\dots,n \big\},   
\end{split} \\
& H^n_0((0,b); \cH) = \big\{f \in H^n((0,b); \cH) \, \big| \,  f^{(j)}(0) =  f^{(j)}(b) = 0, \, j =0,1, \dots, n-1\big\}.  
\end{align}

Next, we derive an elementary version of an $\cH$-valued Friedrichs inequality as follows: 
Suppose $f \in H_1([0,b]; \cH)$, then 
\begin{equation}
f(x) = \int_0^x f'(t) \, dt, \quad x \in [0,b], \; f(0) = 0, 
\end{equation}
implies
\begin{equation}
\|f(x)\|_{\cH} \leq \int_0^x \|f'(t)\|_{\cH} \, dt \leq x^{1/2} \bigg(\int_0^x \|f'(t)\|_{\cH}^2\bigg)^{1/2} 
\leq b^{1/2} \|f'\|_{L^2((0,b); \cH)},    \lb{9.34} 
\end{equation}
and hence $\|f(\, \cdot \,)\|_{\cH} \in L^2((0,b))$. Thus, squaring and then integrating 
\eqref{9.34} with respect to $x$ from $0$ to $b$ yields 
\begin{equation}
\|f\|_{L^2((0,b);\cH)} \leq b \|f'\|_{L^2((0,b); \cH)}, \quad f \in H_1([0,b]; \cH).
\end{equation}
Consequently,
\begin{equation}
H_1([0,b]; \cH) = H^1_0((0,b); \cH),
\end{equation}
and iterating this process finally yields
\begin{equation}
H_n([0,b]; \cH) = H^n_0((0,b); \cH), \quad n \in \bbN. 
\end{equation}

\begin{theorem} \lb{t9.2}
Let $n\in\bbN  $, $b \in (0, \infty),$ and define $H_{n}([0,b]; \cH)$ and
$H^n_0((0,b); \cH)$ as above. Then the following items $(i)$--$(iii)$ hold: \\[1mm] 
$(i)$ For each $n \in \bbN$, 
\begin{equation} 
H_{n}([0,b]; \cH) = H^n_0((0,b); \cH)   \lb{9.39}
\end{equation} 
as sets. In particular, 
\begin{equation}
f\in H_{n}([0,b]; \cH) \, \text{ implies } \, f^{(j)}\in L^{2}((0,b); \cH), \quad j=0,1,\dots,n.   \lb{9.40} 
\end{equation} 
In addition, the norms in $H_{n}([0,b]; \cH)$ $($cf.\ \eqref{9.15a}$)$ and $H^n_0((0,b); \cH)$ are equivalent. \\[1mm] 
$(ii)$ Recalling $d(x) = \min \{x, |b-x|\}$, $x \in (0,b)$, one has 
\begin{equation}
\int_{0}^{b}\big\|f^{(n)}(x)\big\|_{\cH}^{2} \, dx > \dfrac{[(2n-1)!!]^{2}
}{2^{2n}}\int_{0}^{b} \dfrac{\|f(x)\|_{\cH}^2}{d(x)^{2n}} \, dx, \quad f\in H^{n}_0((0,b); \cH) \backslash \{0\}.
\end{equation}
$(iii)$ The constant $[(2n-1)!!]^{2}/2^{2n}$ is sharp. 
\end{theorem}
One can follow the special scalar case treated in the proof of Theorem \ref{t8.2} line by line.



\medskip

\noindent
{\bf Acknowledgments.} 
We are indebted to Katie Elliott, Des Evans, Hubert Kalf, Roger Lewis, and Wilhelm Schlag 
for valuable discussions and hints to the literature. We also thank the anonymous referee for a 
careful reading of our manuscript and for very helpful comments. 

 
\end{document}